\documentclass[a4paper]{article}

\usepackage{amssymb} 
\usepackage{mathrsfs} 
\usepackage{amsthm}
\usepackage{amsmath}

\DeclareMathAlphabet{\mathpzc}{OT1}{pzc}{m}{it} 

\usepackage{color}
\usepackage{graphics}
\usepackage{graphicx}

\newtheorem{Thm}{Theorem}[section]
\newtheorem{Cor}[Thm]{Corollary}
\newtheorem{Lem}[Thm]{Lemma}
\newtheorem{Prop}[Thm]{Proposition}

\theoremstyle{definition}
\newtheorem{Rem}[Thm]{Remark}

\theoremstyle{definition}

\theoremstyle{definition}
\newtheorem{Def}[Thm]{Definition}
\newtheorem{Exa}[Thm]{Examples}
\newtheorem{Example}[Thm]{Example}

\theoremstyle{definition} 

\makeatletter
\@addtoreset{equation}{section}
\makeatother

\def\vuoto{\varnothing} 
\newcommand\funzione{\longrightarrow}

\newcommand\enne{\mathbb{N}} 

\providecommand{\opint}[1]{\hspace{0.15ex}\left]#1\right[\hspace{0.15ex}} 
\providecommand{\clint}[1]{\hspace{0.045ex}[#1]} 
\providecommand{\clsxint}[1]{\hspace{0.1ex}\left[#1\right[\hspace{0.15ex}} 
\providecommand{\cldxint}[1]{\hspace{0.15ex}\left]#1\right]} 

\newcommand\erre{\mathbb{R}} 
\newcommand\ci{\mathbb{C}} 
\newcommand\alg{\mathbb{A}} 

\renewcommand\j{\mathbf{j}} 
\renewcommand\k{\mathbf{k}} 
\DeclareMathOperator{\re}{Re} 
\DeclareMathOperator{\im}{Im} 

\DeclareMathOperator{\de}{d \! \hspace{0.13ex}} 
\renewcommand\r{\textsl{r}} 

\newcommand\End{\textsl{End}}
 
\newcommand\norma[2]{\Vert #1\Vert_{#2}} 
\newcommand\Lin{\mathscr{L}}
\newcommand\GL{\mathscr{G\! L}}

\newcommand{\Id}{\ \!\mathsf{Id}}
\newcommand{\A}{\mathsf{A}}
\newcommand{\B}{\mathsf{B}}
\newcommand{\C}{\mathsf{C}}
\newcommand{\R}{\mathsf{R}}
\newcommand{\Q}{\mathsf{Q}}
\newcommand{\T}{\mathsf{T}}
\newcommand{\U}{\mathsf{U}}

\newcommand\s{\mathpzc{s}} 

\newcommand{\su}{\mathbb{S}}
\newcommand{\lra}{\longrightarrow}
\newcommand{\RR}{\mathbb{R}}
\newcommand{\CC}{\mathbb{C}}
\newcommand{\OO}{\Omega}
\newcommand{\mr}{\mathrm}
\newcommand{\mi}{\mathit}
\newcommand{\mscr}{\mathscr}
\newcommand{\ray}{\mscr{R}}
\newcommand{\ce}{\mscr{C}}
\newcommand{\q}{\mathbb{H}}
\newcommand{\oc}{\mathbb{O}}
\newcommand{\mc}{\mathcal}
\newcommand{\ui}{\mr{i}}
\newcommand{\II}{\mc{I}}

\renewcommand\sp{\hspace{2.8ex}} 

\begin{document}


\markboth{left head}{right head} \markright{Operator semigroups over real *-algebras}

\title{Semigroups over real alternative *-algebras: \\ generation theorems and spherical sectorial operators}

\author{Riccardo Ghiloni\footnote{Dipartimento di Matematica, Universit\`a di Trento, Via Sommarive 14, 38123 Povo-Trento (TN), Italy (\texttt{ghiloni@science.unitn.it})} $\;$ and Vincenzo Recupero\footnote{Dipartimento di Scienze Matematiche, Politecnico di Torino, Corso Duca degli Abruzzi 24, 10129 Torino, Italy (\texttt{vincenzo.recupero@polito.it})}}

\date{}


\maketitle

\thispagestyle{empty}


\begin{abstract}
The aim of this paper is twofold. On one hand, generalizing some recent results obtained in the quaternionic setting, but using simpler techniques, we prove the generation theorems for semigroups in Banach spaces whose set of scalars belongs to the class of real alternative *-algebras, which includes, besides real and complex numbers, quaternions, octonions and Clifford algebras. On the other hand, in this new general framework, we introduce the notion of spherical sectorial operator and we prove that a spherical sectorial operator generates a semigroup that can be represented by a Cauchy integral formula. It follows that such a semigroup is analytic in time.

\ 

\begin{footnotesize}
\noindent
\emph{Keywords}: Functions of hypercomplex variables; Functional calculus; Semigroups of linear operators;
Spectrum, resolvent
 \newline
\noindent
\emph{2010 AMS Subject Classification}: 30G35, 47D03, 47A60, 47A10, 
\end{footnotesize}
\end{abstract}




\section{Introduction}

In the recent years, a great deal of attention has been devoted to a rigorous development of a functional calculus for right linear operators defined on Banach spaces (Banach bimodules to be precise) over the real algebra $\mathbb{H}$ of quaternions (see \cite{CoSaSt08, CoSa09, CoSa10, CoGeSaSt, CoGeSaSt2, CoSaSt, GMP}). One of the main motivations has to be found in the study of quantum mechanics in the quaternionic framework (see, e.g., \cite{FiJaScSp, Em, HoBi, Ad}). Indeed, as pointed out in seminal paper \cite{BiNeu}, quantum mechanics may be formulated, not only on complex Hilbert spaces (the classical setting), but also on Hilbert spaces having $\q$ as the set of scalars.

This new quaternionic functional calculus strongly relies on the theory of slice regular functions. The notion of slice regular function was recently introduced in \cite{GeSt} as a generalization to quaternions of the classical concept of holomorphic function of a complex variable. 
Let $\su$ be the subset of $\q$ of square roots of $-1$ and, for each $\j \in \su$, let $\CC_\j$ be the plane of $\q$ generated by $1$ and $\j$. Observe that each $\CC_\j$ is a copy the complex plane. The quaternions has a ``slice complex'' nature, described by the following two properties: $\q=\bigcup_{\j \in \su}\CC_\j$ and $\CC_\j \cap \CC_\k=\RR$ for every $\j,\k \in \su$ with $\j \neq \pm \k$. Let $D$ be a connected open subset of $\CC$ invariant under complex conjugation and let $\OO_D$ be the open subset of $\q$ obtained by rotating $D$ around $\RR$, i.e. $\OO_D=\bigcup_{\j \in \su}D_\j$, where $D_\j:=\{a+b \, \j \in \CC_\j \, | \, a,b \in \RR, a+bi \in D\}$. A function $f:\OO_D \lra \q$ of class $C^1$ is called \emph{slice regular} if, for every $\j \in \su$, its restriction $f_\j$ to $D_\j$ is holomorphic with respect to the complex structures on $D_\j$ and $\q$ defined by the left multiplication by $\j$, i.e. if $\partial f_\j/\partial a+\j \, \partial f_\j/\partial b=0$ on $D_\j$. A major property of slice regular functions is a Cauchy-type integral formula (see \cite{CoGeSa}). If $D$ is bounded and its boundary is piecewise of class $C^1$, and $f$ is slice and extends to a continuous function on the closure of $\OO_D$ in $\q$, then it holds:
\begin{equation}\label{intro:cauchy formula}
f(p)=\frac{1}{2\pi}\int_{\partial D_\j} C_q(p) \, \j^{-1} \de q \ f(q) \qquad \forall p \in \OO_D, \forall \j \in \su,
\end{equation}
where $C_q(p)$ denotes the
\emph{noncommutative Cauchy kernel}
\[
C_q(p) :=(p^2 - 2\re(q) p + |q|^2)^{-1}(\overline{q}-p),
\]
the line integral in \eqref{intro:cauchy formula} being defined in a natural way (see Appendix-Section \ref{S:appendix} in the case $V=\alg$). The unusual fact that the differential $\de q$ appears on the left of $f(q)$ depends on the noncommutativity of $\q$. Notice that, if $p$ and $q$ commute (for example when $p,q$ belong to the same $\CC_\j$), then it turns out that $C_q(p)=(q-p)^{-1}$ and we find again the form of the classical Cauchy kernel for holomorphic functions.

The theory of quaternionic slice regular functions has been extended to octonions in \cite{GeSt-rocky}. The related theory of slice monogenic functions on Clifford algebras $\RR_n$ was introduced in \cite{CoSaSt09}. These function theories were unified, generalized and developed in \cite{GhPe} by means of the concept of slice function on a real alternative *-algebra. A real algebra $\alg$ is \emph{alternative} if it satisfies the following weak associativity condition: $\alpha(\alpha \beta)=(\alpha\alpha)\beta$ and $(\alpha\beta)\beta=\alpha(\beta\beta)$ for every $\alpha,\beta \in \alg$. Evidently, all the real associative *-algebras, like real and complex numbers, quaternions and Clifford algebras, are alternative. The algebra of octonions is a not associative example of real alternative *-algebra.

As observed in 
\cite{CoGeSaSt,GMP}, 
the classical notions of resolvent set and spectrum are not useful in order to define a noncommutative functional calculus. A promising definition of resolvent set is suggested by the Cauchy integral formula \eqref{intro:cauchy formula} and was given for the first time in \cite{CoSaSt08}. 
If $\A$ is a right linear operator on a quaternionic Banach bimodule, then its \emph{spherical resolvent set} is the set of quaternions $q$ such that the operator
\[
\Delta_q(\A) := \A^2 - 2\re(q)\ \!\A + |q|^2 \Id
\]
is bijective and its inverse is bounded. Accordingly, the \emph{spherical spectrum} is the complement of the spherical resolvent set and the \emph{spherical resolvent operator} $\C_q(\A)$ is defined by
\[
\C_q(\A) := \Delta_q(\A)^{-1}\overline{q} - \A\Delta_q(\A)^{-1}
\]
for every $q$ in the spherical resolvent set of $\A$. These new spectral notions are the starting point of a rigorous treatment of the noncommutative functional calculus (see \cite{CoSaSt}) and now provide the basis for a proper application of the spectral theory to quantum mechanics in the quaternionic setting (see \cite{GMP}).

The next natural stage in this analysis is the development of a noncommutative theory of right linear operator semigroups. 
This task was undertaken in \cite{CoSa}, where the functional calculus is used in order to prove the counterpart of the classical generation theorems by Hille-Yosida and by Feller-Miyadera-Phillips for right linear operators acting on quaternionic Banach bimodules.

The first aim of our work is to extend these generation theorems to the general framework of real alternative *-algebras, by using 
a completely different technique, which does not make use of the functional calculus. In fact, by means of a simple lemma, we reduce our generalized generation theorems to the classical commutative case.

Our second purpose is to study a natural extension of the notion of sectorial operator to the general setting of real alternative *-algebras. We introduce the concept of \emph{spherical sectorial operator} and we prove that such an operator generates a semigroup analytic in time, which can be represented by a Cauchy integral formula. In this case, a suitably generalized functional calculus will play a prominent role. The concept of spherical sectorial operator and the analyticity in time of the corresponding semigroups are new even in the quaternionic framework. A spherical sectorial operator on a quaternionic Banach bimodule $X$ is defined as follows. First, we define the function $\arg:\q \setminus \{0\} \lra \clint{0,\pi}$. If $q \in \q \setminus \RR$, then there exist, and are unique, $\j \in \su$, $r \in \opint{0,\infty}$ and $\theta \in \opint{0,\pi}$ such that $q=r e^{\theta\j} \in \mathbb{C}_\j$. We set: $\arg(q):=\theta$. Moreover, we define: $\arg(q):=0$ if $q \in \opint{0,\infty}$ and $\arg(q):=\pi$ if $q \in \opint{-\infty,0}$. Given a right linear operator $\A$ on $X$, we say that $\A$ is a \emph{spherical sectorial operator} if its spherical resolvent set contains a set $\Sigma_\delta$ of the form
\[
\Sigma_\delta=\{q \in \q \setminus \{0\} \, : \, \arg(q)<\pi/2+\delta\}
\]
for some $\delta \in \opint{0,\pi/2}$. We prove that, if $\A$ has this property and satisfies the estimate 
\[
\norma{\C_q(\A)}{} \le \frac{M}{|q|} \qquad \forall q \in \Sigma_\delta
\]
for some $M \ge 0$, then $\A$ is the generator of the semigroup
\begin{equation}\label{intro:e^At}
\T(t) = \frac{1}{2\pi} \int_{\Gamma_\j} \C_\alpha(\A) \, \j^{-1}e^{t\alpha} \de \alpha, \qquad \forall t > 0, 
\end{equation}
where $\j$ is an arbitrarily fixed element of $\su$ and $\Gamma_\j$ is a suitable path of $\CC_\j$, surrounding the spherical spectrum of $\A$. As a consequence, the semigroup $\T(t)$ is analytic in time. We underline that the noncommutative setting prevents from the possibility of applying the classical strategy. A different technique is needed.

The remaining part of the paper is organized as follows. The next section is devoted to some  prelimi\-nary notions. We recall basic concepts and properties concerning real alternative *-algebras and slice regular functions defined over it. Moreover, we give the precise definition of Banach bimodule over such a *-algebra and we study right linear operators acting on such a Banach bimodule, including the main spectral properties of unbounded right linear operators. In Section \ref{S:classical generation theorems}, we recall the statements of the classical generation theorems for semigroups, we use them
in Section \ref{S:str cont semigroups in A-modules} to prove their counterparts in our real alternative *-algebra framework. Finally, in Section \ref{sec:spherical-sect-operators}, we provide the new notion of spherical sectorial operator and we prove that a spherical sectorial operator generates a semigroup that can be represented by a Cauchy integral formula as in \eqref{intro:e^At}. In Appendix-Section \ref{S:appendix}, we specify the meaning of line integral of continuous functions with values in a Banach bimodule. This kind of line integral appears in the mentioned Cauchy integral formula.


\section{Preliminaries: slice regular functions and Banach bimodules over real alternative *-algebras}\label{S:preliminaries}

\subsection{Real alternative *-algebras and slice regular functions over it}

Let us recall some well-known notions and facts regarding real alternative algebras (see \cite{numbers,GuHaSp,schafer}).

A \emph{real alternative algebra with unity} is a non-empty set $\alg$ endowed with an addition $+ : \alg \times \alg \funzione \alg : (\alpha, \beta) \longmapsto \alpha + \beta$, a product 
$\alg \times \alg \funzione \alg : (\alpha, \beta) \longmapsto \alpha \cdot \beta$ and a scalar multiplication
$\Lambda : \erre \times \alg \funzione \alg : (r,\alpha) \longmapsto r \alpha$ such that:
\begin{itemize}
\item[$(\mr{i})$]
$(\alg, +,\Lambda)$ is a \textit{real algebra}, i.e. $\alg$ is a $\erre$-vector space and the product is $\erre$-bilinear: 
\[
(r\alpha + s\beta)\cdot\gamma = r(\alpha \cdot \gamma) + s(\beta \cdot \gamma), \qquad \gamma \cdot (r\alpha + s\beta) = r(\gamma \cdot \alpha) + s(\gamma \cdot \beta)
\]
for every $\alpha, \beta, \gamma \in \alg$ and for every $r, s \in \erre$.
\item[$(\mr{ii})$] $\alg$ is \emph{alternative}, i.e. the mapping $(\alpha, \beta, \gamma) \longmapsto (\alpha \cdot \beta) \cdot \gamma - \alpha\cdot(\beta \cdot \gamma)$ is alternating: its value changes in sign when any two of its arguments are interchanged. This is equivalent to say that $\alpha\cdot(\alpha\cdot \beta)=(\alpha\cdot\alpha)\cdot\beta$ and $(\alpha\cdot\beta)\cdot\beta=\alpha\cdot(\beta\cdot\beta)$ for every $\alpha,\beta \in \alg$.
\item[$(\mr{iii})$] There exists a (unique) element $1$ of $\alg$, called	\emph{unity} of $\alg$, such that $1 \cdot \alpha= \alpha \cdot 1 = \alpha$ for every $\alpha \in \alg$. We understand that $1 \neq 0$, i.e. $\alg \neq \{0\}$.
\end{itemize}
The real algebra $\alg$ is said to be \emph{finite dimensional} if it has finite dimension as a real vector space. If the product is associative (resp. commutative), then $\alg$ is called \emph{associative} (resp. \emph{commutative}). Clearly, every associative real algebra is alternative. Let $\alpha \in \alg$. If $\alpha \neq 0$ and $\alpha \cdot \beta=0$ for some $\beta \in \alg \setminus \{0\}$, then $\alpha$ is called \emph{zero divisor}. The element $\alpha$ is said to be \textit{invertible} if it has a (two-sided) inverse, i.e. if there exists a (unique) element of $\alg$, denoted by $\alpha^{-1}$, such that $\alpha \cdot \alpha^{-1}=\alpha^{-1} \cdot \alpha=1$. The real alternative algebra $\alg$ is called \emph{real division algebra} if every nonzero element of $\alg$ is invertible. If $\alg$ is finite dimensional, then $\alg$ is a real division algebra if and only if it is without zero divisors.

Suppose that $\alg$ is a real alternative algebra. A straightforward consequence of bilinearity of the product is the formula 
$r(\alpha \cdot \beta) = (r\alpha) \cdot \beta = \alpha\cdot(r\beta)$ holding for every $r \in \erre$ and $\alpha, \beta \in \alg$. In this way, identifying $\RR$ with the subalgebra of $\alg$ generated by the unity $1$, we have:
\[
r\alpha = r \cdot \alpha=\alpha \cdot r \qquad \forall r \in \RR, \ \forall \alpha \in \alg.
\]
In what follows, we will omit the dot notation $\, \cdot \,$ for the product and we will write $\alpha \beta$ rather that $\alpha \cdot \beta$.

As usual powers are defined inductively by $\alpha^0 := 1$, $\alpha^n := \alpha \alpha^{n-1}$ for $n \in \enne$, $n \ge 1$.
From the fact that $\alg$ is alternative follows that $\alg$ is \emph{power associative}, i.e.
\[
  \alpha^{n+m} = \alpha^n \alpha^m \qquad \forall \alpha \in \alg, \ \forall n, m \in \enne.
\]

A real linear mapping $\alg \funzione \alg : \alpha \longmapsto \alpha^c$ is called an
\emph{anti-involution} (or a \emph{*-involution}) if
\begin{itemize}
\item $(\alpha^c)^c = \alpha$ for every $\alpha \in \alg$,
\item $(\alpha\beta)^c = \beta^c \alpha^c$ for every $\alpha, \beta \in \alg$,
\item $r^c = r$ for every $r \in \erre$.
\end{itemize}
Endowing $\alg$ with such an anti-involution, we obtain a \emph{real alternative *-algebra with unity}.

Let us recall the notions of quadratic cone and of  slice regular function, taking \cite{GhPe} as reference.

Throughout the remainder of the paper, we assume that
\begin{equation} \label{eq:assumption-1}
\text{\emph{$\alg$ is a finite dimensional real alternative *-algebra with anti-involution $\alpha \longmapsto \alpha^c$}}.
\end{equation} 
Given $\alpha \in \alg$, we define the \emph{trace} $t(\alpha)$ of $\alpha$ and the \emph{(squared) norm} $n(\alpha)$ by setting
\begin{equation}
t(\alpha) := \alpha + \alpha^c \quad \mbox{and} \quad n(\alpha) := \alpha \alpha^c.
\end{equation}
Moreover, we define the \emph{quadratic cone of $\alg$} \[
Q_\alg := \erre \cup \{\alpha \in \alg\ :\ t(\alpha), n(\alpha) \in \erre,\ t(\alpha)^2 - 4n(\alpha) < 0\}
\]
and the set of \emph{square roots of $-1$}
\begin{equation}
  \su_\alg := \{\j \in Q_\alg \, : \, \j^2 = -1\}.
\end{equation}
Finally, we define the \emph{real part $\re(\alpha)$} and the 
\emph{imaginary part $\im(\alpha)$} of an element $\alpha$ of $Q_\alg$ as follows:
\begin{equation}
\re(\alpha):= (\alpha+\alpha^c)/2 \in \RR
\quad \mbox{and} \quad
\im(\alpha):=(\alpha-\alpha^c)/2.
\end{equation}

For every $\j \in \su_\alg$, we denote by $\CC_\j:=\langle 1, \j \rangle \simeq \CC$ the subalgebra of $\alg$ generated by $\j$. Since $\alg$ is assumed to be alternative, if $\su_\alg \neq \emptyset$, then one can prove that the quadratic cone $Q_\alg$ has the following two properties, which describe its ``slice complex'' nature:
\begin{align}
& Q_\alg = \bigcup_{\j \in \mathbb{S}_\alg} \ci_\j, \label{cone = union of planes} \\
& \ci_\j \cap \ci_\k = \erre \qquad \forall \j, \k \in \mathbb{S}_\alg, \ \j \neq \pm\k. \label{eq:intersection}
\end{align}
For every $\alpha \in Q_\alg$, we define the function $\Delta_\alpha:Q_\alg \lra \alg$ by setting
\begin{equation} \label{eq:delta}
\Delta_\alpha(\beta) :=\beta^2 - \beta t(\alpha) + n(\alpha).
\end{equation}
It is immediate to verify that each $\alpha \in Q_\alg$ satisfies the real quadratic equation $\Delta_\alpha(\alpha)=0$. This equality and \eqref{cone = union of planes} justify the name of $Q_\alg$. Two easy consequences of \eqref{cone = union of planes} are the following:
\begin{itemize}
 \item every $\alpha \in Q_\alg \setminus \{0\}$ is invertible in $\alg$ and $\alpha^{-1}=n(\alpha)^{-1}\alpha^c$,
 \item $\alpha^n \in Q_\alg$ for every $\alpha \in Q_\alg$ and $n \in \enne$.
\end{itemize}

In what follows, we assume that
\begin{equation} \label{eq:assumption-2}
\text{\emph{$\su_\alg \neq \emptyset$ and $\alg$ is endowed with a norm $|\cdot|$ such that $|\alpha|=\sqrt{n(\alpha)}$ for every $\alpha \in Q_\alg$}}.
\end{equation}

We assume also that $\alg$ is equipped with the topology induced by $|\cdot|$.

The reader observes that, if $\alpha \in Q_\alg$ belongs to $\CC_\j$ for some $\j \in \su_\alg$, then there exist $a,b \in \RR$ such that $\alpha=a+b\j$ and hence $\alpha^c=a-b\j$ and $|\alpha|=\sqrt{a^2+b^2}$. In particular, we have: $|\alpha\beta|=|\alpha||\beta|$ and $|\alpha^n|=|\alpha|^n$ for every $\alpha,\beta \in \CC_\j$ and $n \in \enne$. 

We now recall some remarkable examples of real algebras satisfying our assumptions \eqref{eq:assumption-1} and \eqref{eq:assumption-2}. 

\begin{Exa} \label{exa:algebras}
$(\mr{i})$ The simplest example is the one of complex numbers $\CC$, equipped with the usual conjugation map as anti-involution and with the usual euclidean norm. It is easy to see that $Q_\CC=\CC$.

$(\mr{ii})$ Consider the real vector space $\RR^4$, identify $\RR$ with the vector subspace $\RR \times \{0\}$ of $\RR^4=\RR \times \RR^3$ and denote by $\{1,i,j,k\}$ the canonical basis of $\RR^4$. Define a product on $\RR^4$ by imposing the relations
\[
i^2=j^2=k^2=-1, \quad ij=-ji=k, \quad jk=-kj=i, \quad ki=-ik=j
\]
and by using the distributivity to extend such an operation to all $q=a+bi+cj+dk \in \RR^4$, $a,b,c,d \in \RR$. Endowing $\RR^4$ with this product, we obtain the real algebra $\q$ of \emph{quaternions}, which is associative, but not commutative. The standard conjugation map $q \longmapsto \overline{q}:=a-bi-cj-dk$ is an anti-involution and the real-valued function $q \longmapsto |q|:=\sqrt{q\overline{q}}$ is the usual euclidean norm. Such a norm is multiplicative: $|pq|=|p||q|$ for every $p,q \in \q$. Since every $q \in \q \setminus \{0\}$ has the inverse $\overline{q} \, |q|^{-2}$, $\q$ is also a real division algebra. We consider $\q$ to be endowed with the conjugation map $q \longmapsto \overline{q}$ and with the euclidean norm. Finally, we remark that $Q_\q=\q$.

$(\mr{iii})$ The real algebra $\oc$ of \emph{octonions}, also called Cayley numbers, can be obtained from $\q$ as follows. Write any element $x$ of $\oc$ as $x=p+q\ell$, where $p,q \in \q$ and $\ell$ is an imaginary unit, i.e. $\ell^2=-1$. Define the addition and the product on $\oc$ by setting
\[
x+y:=(p+r)+(q+s)\ell
\quad \mbox{and} \quad
xy:=(pr-\overline{s}q)+(q\overline{r}+sp)\ell,
\]
where $y=r+s\ell$, $r,s \in \q$. The real algebra $\oc$ constructed in this way has dimension $8$ and is alternative. However, it is neither commutative nor associative. Consider the basis $\{\mr{o}_0,\ldots,\mr{o}_7\}$ of $\oc$, where $\mr{o}_0:=1$, $\mr{o}_1:=i$, $\mr{o}_2:=j$, $\mr{o}_3:=k$, $\mr{o}_4:=\ell$, $\mr{o}_5:=i\ell$, $\mr{o}_6:=j\ell$ and $\mr{o}_7:=k\ell$. It is known that the conjugation map, sending $x=\sum_{h=0}^7a_h\mr{o}_h$ into $\overline{x}=a_0-\sum_{h=1}^7a_h\mr{o}_h$, is an  anti-involution and the real-valued function $x \longmapsto |x|:=\sqrt{x\overline{x}}$ coincides with the euclidean norm $(\sum_{h=0}^7a_h^2)^{1/2}$. Also in this case, the norm $|\cdot|$ is multiplicative and every nonzero element $x$ has inverse $\overline{x} \, |x|^{-2}$. It turns out that $\oc$ is a real division algebra. We endowed this algebra with the mentioned conjugation map and with the euclidean norm. It follows that $\su_\oc=\oc$.

$(\mr{iv})$ Let $n$ be a positive integer and let $\mc{P}(n)$ be the family of all subsets of $\{1,\ldots,n\}$. Identify $\RR$ with the vector subspace $\RR \times \{0\}$ of $\RR^{2^n}=\RR \times \RR^{2^n-1}$ and denote by $\{e_K\}_{K \in \mc{P}(n)}$ the canonical basis of $\RR^{2^n}$, where $e_{\emptyset}:=1$. For simplicity, if $K=\{k\}$ for some $k$, then we set $e_k:=e_K$. We define a product on $\RR^{2^n}$ by requiring associativity, distributivity and the following relations:
\[
e_k^2=-1, \quad e_ke_h=-e_he_k, \quad e_K=e_{k_1} \cdots e_{k_s}
\]
for every $k,h \in \{1,\ldots,n\}$ with $k \neq h$, and for every $K \in \mc{P}(n) \setminus \{\emptyset\}$ such that $K=\{k_1,\ldots,k_s\}$ with $k_1<\ldots<k_s$. The real algebra obtained endowing $\RR^{2^n}$ with this product is called \emph{Clifford algebra $\RR_n=\mathit{C}\ell_{0,n}$ of signature $(0,n)$}. Evidently, it is associative, but not commutative if $n \geq 2$. Observe that $\RR_1$ and $\RR_2$ are isomorphic to $\CC$ and $\q$, respectively. If $n \geq 3$, a new phenomenon appears: $\RR_n$ has zero divisors as $1-e_{\{1,2,3\}}$. In fact, we have: $(1-e_{\{1,2,3\}})(1+e_{\{1,2,3\}})=0$. The \emph{Clifford conjugation} is an anti-involution of $\RR_n$ defined by setting
\[\textstyle
\overline{x}:=\sum_{K \in \mc{P}(n)}(-1)^{|K|(|K|+1)/2}a_Ke_k \quad \mbox{if } x=\sum_{K \in \mc{P}(n)}a_Ke_k \in \RR_n, \; a_K \in \RR,
\]
where $|K|$ indicates the cardinality of the set $K$. One can prove that the euclidean norm $|x|$ of $x$, i.e. $(\sum_{K \in \mc{P}(n)}a_K^2)^{1/2}$, coincides with $\sqrt{x \overline{x}}$ for every $x \in Q_{\RR_n}$. Actually, there exists another norm $|\cdot|_{\mathit{C}\ell}$ on $\RR_n$, called \emph{Clifford operator norm}, such that $|x|_{\mathit{C}\ell}=\sqrt{x \overline{x}}$ for every $x \in Q_{\RR_n}$. Such a norm is defined by setting
\[
|x|_{\mathit{C}\ell}:=\sup\{|xa| \in \RR \, : \, |a|=1\}.
\]
Unlike the euclidean norm $|\cdot|$, the Clifford operator norm $|\cdot|_{\mathit{C}\ell}$ has the following submultiplicative property:
\begin{equation} \label{eq:submultiplicative}
|xy|_{\mathit{C}\ell} \leq |x|_{\mathit{C}\ell}|y|_{\mathit{C}\ell} \qquad \forall x,y \in \RR_n.
\end{equation}
We consider $\RR_n$ endowed with the Clifford conjugation and with the Clifford operator norm. If $n \geq 3$, then the quadratic cone $Q_{\RR_n}$ of $\RR_n$ is the proper closed subset of $\RR_n$ described by the following equations:
\[
x_K=0 \quad \mbox{and} \quad x \bullet (xe_K)=0 \quad \text{for every $K \in \mc{P}(n)$ with $e_K \neq 1$ and $e_K^2=1$},
\]
where $\bullet$ denotes the standard scalar product on $\RR_n=\RR^{2^n}$. {\tiny$\blacksquare$}
\end{Exa}

It is worth recalling that $Q_\alg=\alg$ if and only if $\alg$ is isomorphic to one of the division *-algebras $\CC$, $\q$ or $\oc$ described above (see \cite{numbers}, \S8.2.4 and \S9.3.2).

We are now in position to recall the notions of slice and slice regular functions.

Let $D$ be a subset of $\CC$, invariant under the complex conjugation $z=a+bi \longmapsto \overline{z}=a-bi$. Define
\[
\OO_D:=\{a+b\j \in Q_\alg \, : \, a,b \in \RR, \, a+bi \in D, \, \j \in \su_\alg\}.
\]
A subset of $Q_\alg$ is said to be \emph{circular} if it is equal to $\OO_D$ for some set $D$ as above. 

Suppose now that $D$ is open in $\CC$, not necessarily connected. Thanks to \eqref{cone = union of planes} and \eqref{eq:intersection}, $\OO_D$ is a relatively open subset of $Q_\alg$.

Consider the complexification $\alg_\CC:=\alg \otimes _\RR \CC$ of $\alg$. The real algebra $\alg_\CC$ can be described as follows. Its elements can be written in the following form: $w=\alpha+\beta\ui$, where $\alpha,\beta \in \alg$ and $\ui$ is an imaginary unit. The sum and the product of $\alg_\CC$ are given by setting
\[
(\alpha+\beta\ui)+(\alpha'+\beta'\ui)=(\alpha+\alpha')+(\beta+\beta')\ui, \quad (\alpha+\beta\ui)(\alpha'+\beta'\ui)=(\alpha\alpha'-\beta\beta')+(\alpha\beta'+\beta\alpha')\ui.
\]

\begin{Def}
A function $F=F_1+F_2\ui:D \lra \alg_\CC$ is called \emph{stem function} if the pair $(F_1,F_2)$ is even-odd w.r.t. the imaginary part of $z \in D$, i.e. $F_1(\overline{z})=F_1(z)$ and $F_2(\overline{z})=-F_2(z)$ for every $z \in D$.

The stem function $F=F_1+F_2\ui$ on $D$ induces a \emph{(left) slice function} $\II(F):\OO_D \lra \alg$ on $\OO_D$ as follows. Let $\alpha \in \OO_D$. By \eqref{cone = union of planes}, there exist $a,b \in \RR$ and $\j \in \su_\alg$ such that $\alpha=a+b\j$. Then we set:
\[
\II(F)(\alpha):=F_1(z)+\j \, F_2(z), \quad \text{where $z=a+bi \in D$}.
\]
\end{Def}

The reader observes that the definition of $\II(F)$ is well-posed. In fact, if $\alpha \in \OO_D \cap \RR$, then $a=\alpha$, $\beta=0$ and $\j$ can be choosen arbitrarily in $\su_\alg$. However, $F_2(z)=0$ and hence $\II(F)(\alpha)=F_1(\alpha)$, independently from the choice of $\j$. If $\alpha \in \OO_\alg \setminus \RR$, then $\alpha$ has the following two expressions:
\[
\alpha=a+b\j=a+(-b)(-\j),
\]
where $a=\re(\alpha)$, $b=|\im(\alpha)|$ and $\j=\im(\alpha)/|\im(\alpha)|$. Anyway, if $z:=a+bi$, we have:
\[
\II(F)(a+(-b)(-\j))=F_1(\overline{z})+(-\j)F_2(\overline{z})=F_1(z)+(-\j)(-F_2(z))=F_1(z)+\j \, F_2(z)=\II(F)(a+b\j).
\]

It is important to observe that every slice function $f:\OO_D \lra \alg$ is induced by a unique stem function $F=F_1+F_2\ui$. In fact, it is easy to verify that, if $\alpha=a+b\j \in \OO_D$ and $z=a+bi \in D$, then
\[
F_1(z)=(f(\alpha)+f(\alpha^c))/2
\quad \mbox{and} \quad
F_2(z)=-\j \, (f(\alpha)-f(\alpha^c))/2.
\]

Let us introduce a subclass of slice functions, which will play an important role in the next sections.  

\begin{Def} \label{eq:real-slice}
Let $F=F_1+F_2\ui:D \lra \alg_\CC$ be a stem function on $D$. The slice function $f=\II(F)$ induced by $F$ is said to be \emph{real} if both the components $F_1$ and $F_2$ of $F$ are real-valued.
\end{Def}

One can prove that the slice function $f$ is real if and only if $f(\OO_D \cap \CC_\j) \subseteq \CC_\j$ for every $\j \in \su_\alg$.

Let $t \in \RR$ and let $\mr{Exp}_t:\CC \lra \alg_\CC$ be the function defined by $\mr{Exp}_t(a+bi):=e^{ta}\cos(b)+e^{ta}\sin(b)\ui$. Such a function is a stem function, which induces the following real slice function $\mr{exp}_t:Q_\alg \lra \alg$:
\begin{equation} \label{eq:exp}
\mr{exp}_t(\alpha)=e^{ta}\cos(b)+\j \, e^{ta}\sin(b)=e^{t\alpha} \qquad \forall \alpha=a+b\j \in Q_\alg,
\end{equation}
where $e^{t\alpha}:=\sum_{n \in \enne}(t\alpha)^n/n!$.

In general, the pointwise product of slice functions is not a slice function. However, if $F=F_1+F_2\ui$ and $G=G_1+G_2\ui$ are stem functions, then it is immediate to see that their pointwise product
\[
FG=(F_1G_1-F_2G_2)+(F_1G_2+F_2G_1)\ui
\]
is again a stem function. In this way, we give the following definition.

\begin{Def}
Let $f=\II(F)$ and $g=\II(G)$ be two slice functions on $\OO_D$. We define the \emph{slice product $f \cdot g$} as the slice function $\II(FG)$ on $\OO_D$.
\end{Def}

It is easy to see that, if $f$ is a real slice function and $g$ is an arbitrary slice function, then their slice product $f \cdot g$ coincides with the pointwise one. In this case, we simply write $fg$ in place of $f \cdot g$.

Our next aim is to recall the concept of slice regular function, which generalizes the notion of holomorphic function from $\CC$ to any real alternative *-algebra like $\alg$.

Let $F=F_1+F_2\ui:D \lra \alg_\CC$ be a stem function whose components $F_1,F_2:D \lra \alg$ are of class $C^1$. Here we assume that $\alg$ is equipped with the standard structure of $C^1$-manifold, as a finite dimensional real vector space.

\begin{Def} \label{def:slice-regular}
The slice function $f=\II(F):\OO_D \lra \alg$ is called \emph{slice regular} if $F=F_1+F_2\ui$ satisfies the following Cauchy-Riemann equations:
\[
\frac{\partial F_1}{\partial a}=\frac{\partial F_2}{\partial b}
\quad \mbox{ and } \quad
\frac{\partial F_1}{\partial b}=-\frac{\partial F_2}{\partial a} \quad \; (z=a+bi \in D).
\]
\end{Def} 

Significant examples of slice regular functions are the polynomials with right coefficients in $\alg$ and, more generally, the convergent power series of the form $\sum_{n \in \enne}\alpha^na_n$ with coefficients $a_n$ in $\alg$. In particular, the exponential function $\mr{exp}_t$ is slice regular for every $t \in \RR$.

We remark that, if $D$ is connected, the notion of slice regular function on $\OO_D$ given in Definition \ref{def:slice-regular} is equivalent to the one given in the Introduction (see \cite[Proposition 8]{GhPe} and \cite[Theorem 2.4]{GhPe-gl}).

Let $\alpha=a+b\j \in Q_\alg$ and let $\su_\alpha:=\{a+b\, \k \in Q_\alg \, : \, \k \in \su_\alg\}$. The function $\Delta_\alpha:Q_\alg \lra \alg$ defined in \eqref{eq:delta} has $\su_\alpha$ as zero set and is real slice. It follows that the function from $Q_\alg \setminus \su_\alpha$ to $Q_\alg$, sending $\beta$ into $\Delta_\alpha(\beta)^{-1}$, is well-defined and real slice. In this way, we can define the slice function $C_\alpha:Q_\alg \setminus \su_\alpha \lra \alg$ by setting
\[
C_\alpha(\beta):=\Delta_\alpha(\beta)^{-1}(\alpha^c-\beta).
\]
Observe that
\[
C_\alpha(\beta) = (\alpha-\beta)^{-1} \qquad \forall \alpha, \beta \in \mathbb{C}_\j,\ \alpha \neq \beta, \ \alpha \neq \beta^c.
\]

Let $\Upsilon_\alg:=\{(\alpha,\beta) \in Q_\alg \times Q_\alg \, : \, \beta \not\in \su_\alpha\}$. The function $\Upsilon_\alg \lra Q_\alg:(\alpha,\beta) \longmapsto C_\alpha(\beta)$ is called \emph{Cauchy kernel for slice regular functions on $\alg$}. This nomenclature is justified by the next result proved in Corollary 28 of \cite{GhPe} (see also \cite{GhPe-volume-cauchy}).

\begin{Thm}[Slice Cauchy formula]
Let $f=\II(F):\OO_D \lra \alg$ be a slice regular function. Suppose that the boundary $\partial D$ of $D$ is piecewise $C^1$, and the components $F_1$ and $F_2$ of $F=F_1+F_2\ui$ are of class $C^1$ and admit continuous extension on the closure of $D$ in $\CC$. Choose $\j \in \su_\alg$ and define $D_\j:=\OO_D \cap \CC_\j$. Then we have:
\begin{equation} \label{eq:cauchy}
f(\beta)=\frac{1}{2\pi} \int_{\partial D_\j} C_\alpha(\beta) \, \j^{-1} \de \alpha \, f(\alpha)
\end{equation}
for every $\beta \in \OO_D$ if $\alg$ is associative or for every $\beta \in D_\j$ if $\alg$ is not associative, where $\partial D_\j$ is the piecewise $C^1$ boundary of $D_\j$ in $\CC_\j$. If $f$ is real slice, then the formula is valid for every $\beta \in \OO_D$ also when $\alg$ is not associative.
\end{Thm}

The precise meaning of the line integral contained in the above Cauchy integral formula is specified in the Appendix-Section \ref{S:appendix}.

If we apply this formula to the real slice function $\exp_t$ defined in \eqref{eq:exp}, then we obtain
\[
e^{t\beta}=\frac{1}{2\pi} \int_{\partial D_\j} C_\alpha(\beta) \, \j^{-1} e^{t\alpha} \de \alpha \qquad \forall t \in \RR, \forall \beta \in \OO_D, \forall \j \in \su_\alg.
\]
We stress that the associativity of $\alg$ is not necessary for the validity of the latter equality.


\subsection{Banach $\alg$-bimodules}

The following definitions summarize the notions of vector space on the quaternions or on Clifford algebras, and that of vector space on the octonions, dealt with in \cite{CoSaSt} and \cite{GoHo}, respectively.

The reader reminds that \emph{we are working with a real algebra $\alg$ satisfying \eqref{eq:assumption-1} and \eqref{eq:assumption-2}}.

\begin{Def}\label{D:left A-modulo}
Let $(V,+)$ be an abelian group. We say that $V$ is a \emph{(left) $\alg$-module} if it is endowed with a left scalar multiplication 
$\alg \times V \funzione V : (\alpha,v) \longmapsto \alpha v$ such that
\begin{alignat}{5}
  & \alpha(u+v) = \alpha u + \alpha v, 	& \qquad 	& \forall u, v \in V,	& \quad	& \forall \alpha \in \alg, \\
  & (\alpha+\beta)u = \alpha u + \beta u, 	& \qquad 	& \forall u \in V, 		& \quad 	& \forall \alpha, \beta \in \alg, \\   
  & 1u = u, 						&\qquad 	& \forall u \in V, \\  
  & r(su) = (rs)u 					& \qquad 	& \forall u \in V, 		& \quad 	& \forall r, s \in \erre. 
     \label{ass del prod per reali-sx}
\end{alignat}
If $\alg$ is associative it is also required that
\begin{alignat}{5}
  & \alpha(\beta u) = (\alpha\beta)u,	& \qquad & \forall u \in V,	& \quad	& \forall \alpha, \beta \in \alg. 
      \label{ass del prod per scalari-sx}	
\end{alignat}
An abelian subgroup $W$ of $V$ is a \emph{(left) $\alg$-submodule} if $\alpha u \in W$ whenever $u \in W$ and $\alpha \in \alg$.
\end{Def}

If $\alg$ is a field we obtain the classical notion of vector space. The definition of right $\alg$-module is completely analogous:

\begin{Def}\label{D:right A-modulo}
Let $(V,+)$ be an abelian group. We say that $V$ is a \emph{right $\alg$-module} if it is endowed with a right scalar multiplication
$V \times \alg \funzione V : (\alpha,v) \longmapsto \alpha v$ such that
\begin{alignat}{5}
  & (u+v)\alpha = u\alpha + v\alpha, 		& \qquad 	& \forall u, v \in V,	& \quad	& \forall \alpha \in \alg, \\
  & u(\alpha+\beta) = u\alpha + u\beta , 	& \qquad 	& \forall u \in V, 		& \quad 	& \forall \alpha, \beta \in \alg, \\   
  & u 1= u, 						&\qquad 	& \forall u \in V, \\  
  & (ur)s = u(rs) 					& \qquad 	& \forall u \in V, 		& \quad 	& \forall r, s \in \erre. 
     \label{ass del prod per reali-dx}
\end{alignat}
If $\alg$ is associative it is also required that
\begin{alignat}{5}
  & (u\alpha)\beta  = u(\alpha\beta),	& \qquad	& \forall u \in V,	& \quad	& \forall \alpha, \beta \in \alg. 
      \label{ass del prod per scalari-dx}	
\end{alignat}
An abelian subgroup $W$ of $V$ is a \emph{right $\alg$-submodule} if $u\alpha \in W$ whenever $u \in W$ and $\alpha \in \alg$.
\end{Def}

Now we can give the notion of $\alg$-bimodule.

\begin{Def}\label{D:2sided A-modulo}
Let $(V,+)$ be an abelian group. We say that $V$ is a \emph{$\alg$-bimodule} if it is endowed with two scalar multiplications 
$\alg \times V \funzione V : (\alpha,v) \longmapsto \alpha v$ and $V \times \alg \funzione V : (v,\alpha) \longmapsto v \alpha$ such that $V$ is both a left $\alg$-module and a right $\alg$-module and 
\begin{equation}\label{prod per reali}
  r u = u r \qquad \forall u \in V, \quad \forall r \in \erre. 
\end{equation}
If $\alg$ is associative, it is also required that
\begin{equation}\label{compatibilita dei prod per scalari}
 \alpha(u\beta) = (\alpha u)\beta \qquad \forall u \in V, \quad \forall \alpha, \beta \in \alg. 	
\end{equation}
An abelian subgroup $W$ of $V$ is a \emph{sub-bimodule} if it is both a left and a right $\alg$-submodule of $V$.
\end{Def}

In an associative framework, the previous definitions are standard and can be found, e.g., in 
\cite[Chapter 1, Section 2, p. 26-28]{AnFu}, where $\alg$ is an associative ring and does not contain necessarily the real numbers. In our case, $\alg$ contains a copy of $\erre$, but it is not necessarily associative, therefore it is natural to require conditions \eqref{ass del prod per reali-sx}, \eqref{ass del prod per reali-dx} and \eqref{prod per reali}. These conditions easily yield the following lemma.

\begin{Lem}
If $V$ is an $\alg$-bimodule then
\begin{equation}
  r(us) = (r u)s \qquad \forall  r, s \in \erre. \label{compatibilita dei prod per scalari reali}
\end{equation}
Moreover, with the notations of Definition \ref{D:2sided A-modulo}, the abelian group $(V,+)$, endowed with the (left) scalar multiplication 
$\erre \times V \funzione V:(r,v) \longmapsto r v$, is a (left) $\erre$-vector space.
\end{Lem}

In \cite{AnFu} it is suggested a self-explanatory notation which is useful when we consider different sets of scalars simultaneously: if $V$ is an abelian group then
\begin{align}
  & \text{$_\alg V$ means that $V$ is considered as a left $\alg$-module}, \\
  & \text{$V_\alg$ means that $V$ is considered as a right $\alg$-module}, \\
  & \text{$_\alg V_\alg$ means that $V$ is considered as a $\alg$-bimodule}.
\end{align}
However we will not use the last notation and we will simply say that $V$ is an $\alg$-bimodule.

The following is a natural generalization of the notion of norm in a vector space.

\begin{Def}\label{D:norma su V}
Let $V$ be an $\alg$-bimodule. A function 
$\norma{\cdot}{} : V  \funzione \clsxint{0,\infty}$ is called a \emph{norm on $V$} if 
\begin{align}
  & \norma{u}{} = 0 \ \Longleftrightarrow \ u = 0, \\
  & \norma{u + v}{} \le \norma{u}{} + \norma{v}{} \qquad \forall u, v \in V, \\
  & \norma{\alpha u}{} \leq |\alpha| \, \norma{u}{} \, \mbox{ and } \, \norma{u \alpha}{} \leq |\alpha| \,  \norma{u}{} \label{pos omog su V}
      \qquad \forall u \in V, \quad \forall \alpha \in \alg,\\
  & \norma{\alpha u}{}=\norma{u \alpha}{}=|\alpha| \, \norma{u}{} \qquad \forall u \in V, \quad \forall \alpha \in Q_\alg. \label{eq:homog}
\end{align}
Since $\RR \subset Q_\alg$, \eqref{eq:homog} implies that $\norma{\cdot}{}$ is a norm on $_\RR V$ in the usual sense. We call the $\alg$-bimodule $V$, equipped this kind of norm, \emph{normed $\alg$-bimodule} and we endow $V$ with the topology induced by the metric $d : V \times V \funzione \clsxint{0,\infty} : (u,v) \longmapsto \norma{u-v}{}$, i.e. with the metric induced by 
$\norma{\cdot}{}$ as a norm on $_{\erre}V$. Finally, we say that $V$ is a \emph{Banach $\alg$-bimodule} if the norm 
$\norma{\cdot}{}$ is complete, i.e. if a sequence $(u_n)$ in $V$ converges whenever $\lim_{n,m \to \infty} \norma{u_n- u_m}{} = 0$.
Equivalently, $V$ is a Banach $\alg$-bimodule if and only if $_\erre V$ is a Banach space.
\end{Def}

We underline that, if $\alg$ is associative, then \eqref{eq:homog} follows from \eqref{pos omog su V}. In fact, if $\alpha=0$, then \eqref{eq:homog} is evident. Otherwise, if $\alpha \in Q_\alg \setminus \{0\}$ and $u \in V$, then  \eqref{eq:assumption-2} implies that
\[
\norma{\alpha u}{} \leq |\alpha| \, \norma{u}{}=|\alpha| \, \norma{\alpha^{-1}\alpha u}{} \leq |\alpha| \, |\alpha^{-1}| \, \norma{\alpha u}{}=|\alpha\alpha^{-1}| \, \norma{\alpha u}{}=|1| \, \norma{\alpha u}{}=\norma{\alpha u}{}
\]
and hence $\norma{\alpha u}{}=|\alpha| \, \norma{u}{}$. Similarly, we infer that $\norma{u\alpha}{}=|\alpha| \,\norma{u}{}$.

We conclude this subsection presenting simple examples of Banach bimodule.

\begin{Example}
Let $n$ be an integer $\geq 2$ and let $\RR_n$ be the Clifford algebra of signature $(0,n)$, endowed with the Clifford conjugation and with the Clifford norm $|\cdot|_{C\ell}$ (see Examples \ref{exa:algebras},$(\mr{iv})$). 
Consider a positive integer $m$ and endow the set $V=(\RR_n)^m$ with the natural componentwise addition and left and right $\RR_n$-scalar multiplications
\begin{align}
&(x_1,\ldots,x_m)+(y_1,\ldots,y_m)=(x_1+y_1,\ldots,x_m+y_m), \label{eq:add}\\
&\alpha(x_1,\ldots,x_m)=(\alpha x_1,\ldots,\alpha x_m), \; (x_1,\ldots,x_m)\alpha=(x_1\alpha,\ldots,x_m\alpha), \label{eq:scalar-mult}
\end{align}
and with the norm
\[
\norma{(x_1,\ldots,x_m)}{}=\left(|x_1|_{C\ell}^2+\ldots+|x_m|_{C\ell}^2\right)^{1/2}.
\]
It is easy to verify that $V$ is a Banach $\RR_n$-bimodule. The reader observes that, if $n \geq 3$, there always exist $\alpha \in \RR_n$ and $u \in V$ such that the inequalities in \eqref{pos omog su V} are not equalities. In fact, if $\alpha=1-e_{\{1,2,3\}}$ and $u=(1+e_{\{1,2,3\}},0,\ldots,0)$, then $\alpha u=0=u\alpha$ and hence $\norma{\alpha u}{}=0=\norma{u\alpha}{}$, while $|\alpha| \, \norma{u}{}\neq 0$.

Similarly, one can define a structure of Banach $\oc$-bimodule on $\oc^m$ by means of formulas \eqref{eq:add}-\eqref{eq:scalar-mult} and of the euclidean norm $\norma{(x_1,\ldots,x_m)}{}=\left(|x_1|^2+\ldots+|x_m|^2\right)^{1/2}$.  
\end{Example}


\subsection{Right linear operators}

Now we introduce the concept of right linear operators acting on $\alg$-bimodules.

\begin{Def}
Let $V$ be an $\alg$-bimodule and let $D(\A)$ be a right $\alg$-submodule of $V$. We say that $\A : D(\A) \funzione V$ is \emph{right linear} if it is additive and
\begin{equation}
 \A(u\alpha) = \A(u) \alpha \qquad \forall u \in V, \quad \forall \alpha \in \alg.
\end{equation}
As usual, the notation $\A v$ is often used in place of $\A(v)$. We use the symbol $\End^{\r}(V)$ to denote the set of right linear operators $\A$ with $D(\A) = V$. The identity operator is right linear and is denoted by $\Id_V$ or simply by $\Id$ 
if no confusion may arise. Moreover, if $V$ is a normed $\alg$-bimodule, then we say that 
$\A : D(\A) \funzione V$ is \emph{closed} if its graph is closed in $V \times V$. As in the classical theory, we set 
$D(\A^2) := \{x \in D(\A) \, : \, \A x \in D(\A)\}$. 
\end{Def}

Observe that, if $D(\A)$ is a right $\alg$-submodule of $V$ and if $\A : D(\A) \funzione V$ is right linear, then $D(\A)$ is a 
(left) $\erre$-vector subspace of $_\erre V$ and $\A$ is (left) $\erre$-linear on $_\erre V$. In particular, 
$\End^\r(V)$ is contained in $\End(_\erre V)$, the set of (left) $\erre$-linear operators from $_\erre V$ into itself.

Let us also recall the following definition (see e.g. \cite[Chapter 1, p. 55-57]{AnFu}).

\begin{Def}
Let $V$ be an $\alg$-bimodule, let $D(\A)$ be a right $\alg$-submodule of $V$ and let $\alpha \in \alg$. If $\A : D(\A) \funzione V$ is a right linear operator, then we define the mapping 
$\alpha \A : D(\A) \funzione V$ by setting
\begin{equation}
   (\alpha \A)(u) := \alpha \A(u), \qquad \forall u \in V.  \label{funzione per scalare a sx}
\end{equation}  
If $D(\A)$ is also a left $\alg$-submodule of $V$, then we can define $\A\alpha :D(\A) \funzione V$ by setting
\begin{equation}
(\A\alpha)(u):=\A(\alpha u), \qquad \forall u \in V. \label{funzione per scalare a dx}
\end{equation} 
The sum of operators is defined in the usual way.
\end{Def}

We have the following easy result.

\begin{Prop}
Let $V$ be an $\alg$-bimodule. The following assertions hold.
\begin{itemize}
\item[$(\mr{i})$] 
If $\A \in \End^\r(V)$ and $\lambda \in \erre$ then $\lambda\A = \A\lambda \in \End^\r(V)$. Moreover, formula 
\eqref{funzione per scalare a sx} with $\alpha \in \erre$ makes $_\erre \End^\r(V)$ an $\erre$-vector subspace of $\End(_{\erre}V)$. 
\item[$(\mr{ii})$] Assume that $\alg$ is associative. If $\A \in \End^\r(V)$ and $\alpha \in \alg$ then $\alpha\A, \A\alpha \in \End^\r(V)$. Moreover, formulas \eqref{funzione per scalare a sx}-\eqref{funzione per scalare a dx} make $\End^\r(V)$ an $\alg$-bimodule. 
\end{itemize}
\end{Prop}

Notice that, in the general not associative 
setting, the mappings defined in \eqref{funzione per scalare a sx}-\eqref{funzione per scalare a dx} are not right linear. A simple example is the one in which $\A$ is the identity operator on $V=\oc$ and $\alpha=i \in \oc=\alg$. This fact might suggest that we should consider only the case when $\alg$ is associative. However, as we will see in Section \ref{S:str cont semigroups in A-modules}, only the $\erre$-linearity is relevant in proving generation theorems for semigroups in the alternative framework. Associativity does not play a role and in fact one works in the spaces $_\erre \End^\r(V)$ and 
$\End(_\erre V)$.

\begin{Lem}\label{L:inverse is linear}
Let $V$ be an $\alg$-bimodule. The following assertions hold.
\begin{itemize}
 \item[$(\mr{i})$] The composition of two right linear operators is right linear on its domain of definition.
 \item[$(\mr{ii})$] Let $\A : D(\A) \funzione V$ be right linear and let $R(\A)$ be its range. Assume that $\A$ has left inverse $\B : R(\A) \funzione V$, i.e. $\B\A = \Id$. Then $R(\A)$ is a right $\alg$-submodule of $V$ and $\B$ is right linear.
\end{itemize}
\end{Lem}

\begin{proof}
The first part $(\mr{i})$ is trivial. It is also easy to check that $R(\A)$ is a right $\alg$-submodule of $V$. If $y \in R(\A)$ and $\alpha \in \alg$, we are left to prove that $\B(y\alpha) = \B(y) \alpha$. Let $x \in D(\A)$ be such that $\A(x) =\ y$. Then $\A(x\alpha) = \A(x)\alpha = y\alpha$, therefore, taking $\B$ on both left-sides, $x\alpha = \B(y\alpha)$. On the other hand, $x=\B(y)$, thus $\B(y)\alpha = \B(y\alpha)$ and the lemma is proved.
\end{proof}

\begin{Def}
Let $V$ be a normed $\alg$-bimodule with norm $\norma{\cdot}{}$. For every $\A \in \End^\r(V)$, we set
\begin{equation}\label{norma operatoriale}
  \norma{\A}{} := \sup_{u \neq 0} \frac{\norma{\A u}{}}{\norma{u}{}}
\end{equation}
and we define the sets
\begin{gather}
   \Lin^\r(V) := \{\A \in \End^\r(V)\ :\ \norma{\A}{} < \infty\}, \notag \\
   \GL^\r(V) := \{\A \in \Lin^\r(V)\ :\ \exists \A^{-1} \in \Lin^r(V)\}. \notag
\end{gather}
\end{Def}

Observe that $\norma{\A}{}$ can be equivalently defined as the operatorial norm of $\A$ as an element of $\End(_\erre V)$, therefore
\begin{align}\label{altra descrizione di Lr(V)}
  \Lin^\r(V) 
    & = \{\A \in \End(_\erre V) \, : \, A \text{ is right linear}, \ \norma{A}{} < \infty\} \notag \\
    & = \{\A \in \Lin(_\erre V) \, : \, A \text{ is right linear}\},
\end{align}
where $\Lin(_\erre V) = \{A \in \End(_\erre V)\ :\ \norma{\A}{} < \infty\}$ is the usual normed $\RR$-vector space of continuous linear operators on $_\RR V$.

\begin{Prop}
Let $V$ be a normed $\alg$-bimodule with norm $\norma{\cdot}{}$. The following assertions hold.
\begin{itemize}
\item[$(\mr{i})$] $_\erre \Lin^\r(V)$ is an $\erre$-vector subspace of $\Lin(_{\erre}V)$.
\item[$(\mr{ii})$]
  Assume that $\alg$ is associative. If $\A \in \Lin^\r(V)$ and $\alpha \in \alg$ then $\alpha\A, \A\alpha \in \Lin^\r(V)$. Hence
  $\Lin^\r(V)$ is an $\alg$-bimodule. 
\end{itemize}
\end{Prop}

In general, $\Lin^\r(V)$ is not an $\alg$-bimodule, since 
\eqref{funzione per scalare a sx}-\eqref{funzione per scalare a dx} do not define right linear operators. Therefore $\norma{\cdot}{}$ is not a norm in the sense of Definition \ref{D:norma su V}, but it is a norm on $_\erre \Lin^\r(V)$. Furthermore, $d(A,B) := \norma{A-B}{}$ defines a metric on $\Lin^\r(V)$, which obviously induces the same topology of the relative topology of $\Lin^\r(V)$ as a subset of 
$\Lin(_\erre V)$ endowed with the operatorial norm $\norma{\cdot}{}$.

The following simple lemma is a key result for Section \ref{S:str cont semigroups in A-modules}.

\begin{Lem}\label{L:L^r chiuso in L}
Let $V$ be a normed $\alg$-bimodule with norm $\norma{\cdot}{}$. The $\erre$-vector subspace $\Lin^\r(V)$ of $\Lin(_{\erre}V)$ is closed with respect to the topology of pointwise convergence and hence with respect to the uniform operator topology of $\Lin(_{\erre}V)$.
\end{Lem}

\begin{proof}
Let $\A \in \Lin(_{\erre}V)$ and let $(\A_n)$ be a sequence in $\Lin^\r(V)$ such that $\A_{n}v \to \A v$ for every $v \in V$. Then \eqref{pos omog su V} yields that $\norma{\A(u \alpha) - \A(u)\alpha}{} \le \norma{\A(u\alpha) - \A_n(u\alpha)}{} + \norma{\A_n u - \A u}{}|\alpha|$ for every $u \in V$, 
$\alpha \in \alg$ and $n \in \enne$. Taking the limit as $n \to \infty$, we get the right linearity of $\A$.
\end{proof}

\begin{Cor}\label{C:A+B invertibile}
Let $V$ be a Banach $\alg$-bimodule with norm $\norma{\cdot}{}$. If $\A \in \GL^\r(V)$ and $\B \in \Lin^\r(V)$ with $\norma{\B}{} < \norma{\A^{-1}}{}^{-1}$, then $\A + \B \in \GL^\r(V)$. In particular, $\GL^\r(V)$ is open in $\Lin^\r(V)$. Moreover,
\begin{equation}\label{altra descrizione di GLr(V)}
  \GL^\r(V) = \{\A \in \Lin^\r(V)\ :\ \emph{\text{$\A$ is bijective}}\} \subseteq \GL(_\erre V),
\end{equation}
where $\GL(_\erre V)$ is the set of invertible operators of $\Lin(_\erre V)$.
\end{Cor}
\begin{proof}
Since $\norma{\A^{-1}\B}{} < 1$, the series $\sum_{n=0}^\infty (-\A^{-1}\B)^n$ converges in $\Lin(_\erre V)$ and its sum is 
$(\Id+\A^{-1}\B)^{-1}$. Every partial sum of this series is an element of $\Lin^\r(V)$, hence Lemma \ref{L:L^r chiuso in L} ensures that $(\Id + \A^{-1}\B)^{-1} \in \Lin^\r(V)$. Thus, as in the classical case, from the equality $\A + \B = \A(\Id + \A^{-1}\B)$, we infer that there exists $(\A + \B)^{-1} = (\Id + \A^{-1}\B)^{-1} \A^{-1} \in \Lin^\r(V)$. Concerning the last statement, if 
$\A \in \Lin^\r(V)$ is bijective then, by Lemma \ref{L:inverse is linear}, $\A^{-1} \in \End^\r(V)$. On the other hand, $\A$ belongs to $\Lin(_\erre V)$ and hence, by the open mapping theorem, we infer that $\A^{-1} \in \Lin(_\erre V)$. Thus we deduce 
\eqref{altra descrizione di GLr(V)} from \eqref{altra descrizione di Lr(V)}. 
\end{proof}

When $\alg$ is associative, it is easily seen that \eqref{funzione per scalare a sx}, \eqref{funzione per scalare a dx} and \eqref{norma operatoriale} make $\Lin^\r(V)$ a normed $\alg$-bimodule in the sense of Definition \ref{D:norma su V}. Of course, we obtain the same topology as the one induced by $\Lin(_\erre V)$. We summarize these facts in the following proposition.

\begin{Prop}\label{P:prop of L^r(V)}
Let $V$ be a normed $\alg$-bimodule with norm $\norma{\cdot}{}$. Assume that $\alg$ is associative. Then formulas \eqref{funzione per scalare a sx}, \eqref{funzione per scalare a dx} and and \eqref{norma operatoriale} make $\Lin^\r(V)$ a normed $\alg$-bimodule. In particular, the operatorial norm $\norma{\cdot}{}$ defined in \eqref{norma operatoriale} is a norm on $\Lin^\r(V)$ in the sense of Definition \ref{D:norma su V}. The topology induced by this norm is the same topology of $\Lin^\r(V)$ as a topological subspace of $\Lin(_{\erre}V)$ endowed with the uniform operator topology.

Furthermore, if $V$ is Banach, the same is true for $\Lin^\r(V)$.  
\end{Prop}

For the sake of completeness, we explicitly state a simple fact that will be repeatedly used in the sequel. 

\begin{Lem}\label{L:continuita invariante} 
If $V$ is a normed $\alg$-bimodule with norm $\norma{\cdot}{}$, then $C(\clsxint{0,\infty};V) = C(\clsxint{0,\infty}; _{\erre}\!V)$ and 
$C(\clsxint{0,\infty};\Lin^\r(V)) = \{f \in C(\clsxint{0,\infty}; \Lin(_{\erre}\!V)) \, : \, f(t) \in \Lin^\r(V) \mbox{ for every } t \ge 0\}$.
\end{Lem}

Now we want to make precise the notion of differentiability for functions of one real variable with values in a normed $\alg$-bimodule. Since the domain is one dimensional and real, differentiability makes sense if we consider the codomain as a normed real vector space. Hence there is nothing really new and we give the following definition.

\begin{Def}
Let $I \subseteq \erre$ be an interval and let $(E,+)$
be an abelian group endowed with a scalar multiplication $\erre \times E \funzione E$ that
makes $_\erre E$ a real vector space. Assume that $\norma{\cdot}{}$ is a norm on $_\erre E$ in the usual sense. A map $f : I \funzione E$ is called \emph{differentiable in $t \in I$} if there exists $\lim_{h \to 0}\frac{1}{h}(f(t+h) - f(t))$ in $_\erre E$, the limit being taken with respect 
to the topology induced by $\norma{\cdot}{}$. Accordingly, if $V$ is a normed $\alg$-bimodule then $f : I \funzione V$ is 
differentiable in $t \in I$ if it is differentiable as a function from $I$ into $_\erre V$, and $g : I \funzione \Lin^r(V)$ is 
differentiable in $t \in I$ if it is differentiable as a function from $I$ into $_\erre \Lin^r(V)$. By Lemma \ref{L:L^r chiuso in L}, we can say that $g : I \funzione \Lin^r(V)$ is differentiable in $t \in I$ if and only if it is differentiable as a function from $I$ into 
$\Lin(_\erre V)$.
\end{Def}


\subsection{Spectral theoretical notions}

In this section, we assume that
\begin{center}
\emph{$X$ is a Banach $\alg$-bimodule with norm $\norma{\cdot}{}$}.
\end{center}
We remind the reader that $\alg$ is a real algebra satisfying \eqref{eq:assumption-1} and \eqref{eq:assumption-2}.

Let us start by recalling the basic spectral notions for operators in real or complex Banach spaces, which makes sense also for Banach $\alg$-bimodules.

\begin{Def}\label{D:classical spectral notions}
Let $D(\A)$ be a right $\alg$-submodule of $X$ and let $\A : D(\A) \funzione X$ be a closed right linear operator. We define:
\begin{itemize}
\item[(i)] $\sigma(\A) := \{\alpha \in Q_\alg \, : \, \alpha \Id - \A \text{ is not bijective}\}$ (the \emph{spectrum} of $\A$).
\item[(ii)] $\rho(\A) := Q_\alg \setminus \sigma(\A)$ (the \emph{resolvent set of $\A$}).
\item[(iii)] $\R_\alpha(\A) := (\alpha \Id - \A)^{-1} : X \funzione D(\A) \quad \mbox{if } \alpha \in \rho(\A)$ (the \emph{resolvent operator of $\A$ at $\alpha$}).
\end{itemize}
\end{Def}

If $\alg$ is associative then the resolvent operator is a bounded right linear operator, by virtue of \eqref{compatibilita dei prod per scalari}, of the classical closed graph theorem, and of Lemma \ref{L:inverse is linear}. Notice that if $\alpha = \lambda$ is real, then $\R_\lambda(\A)$ is a bounded right linear operator also in the not associative 
case. 
By arguing exactly as in the classical case (see, e.g., \cite[Lemma 6, Section VII.3, p. 568]{DuSc}), from the definition of resolvent operator, we can easily deduce the following \emph{resolvent equation}.

\begin{Prop}
Given a closed right linear operator $\A : D(\A) \funzione X$, we have:
\begin{equation}\label{resolvent equation}
  \R_\lambda(\A) - \R_\mu(\A) = (\mu - \lambda)\R_\lambda(\A) \R_\mu(\A) \qquad \forall \lambda, \mu \in \rho(\A) \cap \erre.
\end{equation}
\end{Prop}

When $\alg$ is the division algebra of quaternions, it has been shown that the generalizations of the classical spectral notions given in Definition \ref{D:classical spectral notions} are not useful tools (see 
\cite[Sect. 4]{CoGeSaSt} and \cite[Sect. 4]{GMP}). The new noncommutative Cauchy kernel $C_\alpha(\beta)=\Delta_\alpha(\beta)^{-1}(\alpha^c-\beta)$ suggests to systematically replace $\alpha \Id - \A$ with $\Delta_\alpha(\A)$. This was done for the first time in  \cite[Definition 4.6., p. 835]{CoGeSaSt} in the quaternionic setting, by introducing the notions of quaternionic spherical spectrum, resolvent set and resolvent operator.

Let us extend these notions to all real alternative *-algebras. 

\begin{Def}\label{D:spherical spectral notions}
Let $D(\A)$ be a right $\alg$-submodule of $X$ and let $\A : D(\A) \funzione X$ be a closed right linear operator. Given $\alpha \in Q_\alg$, we define the right linear operator $\Delta_\alpha(\A) : D(\A^2) \funzione X$ by setting
\begin{equation}
  \Delta_\alpha(\A) := \A^2 - 2\re(\alpha) \,\A + |\alpha|^2 \Id.
\end{equation}
The \emph{spherical resolvent set $\rho_\s(\A)$ of $\A$} and the \emph{spherical spectrum $\sigma_\s(\A)$ of $\A$} are the circular subsets of $Q_\alg$ defined as follows:
\begin{equation}
\rho_\s(\A) := \{\alpha \in Q_\alg \, : \, \text{$\Delta_\alpha(\A)$ is bijective, $\Delta_\alpha(\A)^{-1} \in \Lin^\r(X)$}\}
\end{equation}
and
\begin{equation}
  \sigma_\s(\A) :=Q_\alg \setminus \rho_\s(\A).
\end{equation}
For every $\alpha \in \rho_\s(\A)$, we define the operators $\Q_\alpha(\A) \in \Lin^\r(X)$ and 
$\C_\alpha(\A) \in \End(_\erre X)$ by setting
\begin{equation}
  \Q_\alpha(\A) := \Delta_\alpha(\A)^{-1} 
\end{equation}
and
\begin{equation}\label{resolvent operator}
\C_\alpha(\A) := \Q_\alpha(\A) \alpha^c - \A \Q_\alpha(\A).
\end{equation}
The operator $\C_\alpha(\A)$ is called \emph{spherical resolvent operator of $\A$ at $\alpha$}.
\end{Def}

\begin{Rem}
The reader observes that the range of $\Q_\alpha(\A)$ is obviously $D(\A^2)$ for every $\alpha \in Q_\alg$. Moreover, $\Delta_\alpha(\A) = \A^2 - \A 2 \re(\alpha) + |\alpha|^2\Id$, because $\re(\alpha)$ is real.
\end{Rem}

We would like to mention that a definition that has some similarities with the spherical spectrum was given in \cite{Kap} in the context of real *-algebras.

Now we collect in the next proposition some significant properties of the operators just defined.

\begin{Prop}
Let $\A : D(\A) \funzione X$ be a closed right linear operator with $D(\A)$ dense in $X$. Then, for every 
$\alpha, \beta \in \rho_\s(\A)$, we have:
\begin{align}
&\label{Q e A commutano}
\Q_\alpha(\A) \A x = \A \Q_\alpha(\A)x \qquad \forall x \in D(\A),\\
&\label{risolvente sul dominio}
\C_\alpha(\A)x = \Q_\alpha(\A)(\alpha^c\Id - \A)x \qquad \forall x \in D(\A),\\
&\label{i Q_lambda commutano}
\Q_\alpha(\A)\Q_\beta(\A) = \Q_\beta(\A)\Q_\alpha(\A)
\end{align}
and also
\begin{align}
&\label{eq:C_alpha}
\C_\alpha(\A) \in \Lin(_\erre X). 
\end{align}
\end{Prop}
\begin{proof}
First we observe that $\Delta_\alpha(\A)$ commutes with $\A$, hence for every $x \in D(\A^2)$ we have that
\begin{align}
 \Q_\alpha(\A)\A x 
   & = \Q_\alpha(\A)\A\Delta_\alpha(\A)\Q_\alpha(\A)x \notag \\
   &   = \Q_\alpha(\A)\Delta_\alpha(\A) \A\Q_\alpha(\A)x
        = \A\Q_\alpha(\A)x, \notag
\end{align}
i.e. \eqref{Q e A commutano} holds. Formula \eqref{risolvente sul dominio} immediately follows. Concerning 
\eqref{i Q_lambda commutano} observe that, by \eqref{Q e A commutano}, $\Delta_\alpha(\A)$ and $\Q_\beta(\A)$ commute, hence
\begin{align}
  \Q_\alpha(\A) \Q_\beta(\A) 
    & = \Q_\beta(\A)\Delta_\beta(\A) \Q_\alpha(\A)  \Q_\beta(\A) \notag \\
    & = \Q_\beta(\A) \Q_\alpha(\A) \Delta_\beta(\A) \Q_\beta(\A) = \Q_\beta(\A) \Q_\alpha(\A). \notag
\end{align}

It is clear 
that 
$\C_\alpha(\A)$ is 
$\erre$-linear and its domain is $X$. Therefore, by the closed graph theorem, we infer that $\C_\alpha(\A)$ is continuous and \eqref{eq:C_alpha} is proved.
\end{proof}

\begin{Lem}\label{L:Calpha - AC = I}
Let $\A : D(\A) \funzione X$ be a closed right linear operator with $D(\A)$ dense in $X$ and let $\alpha \in \alg$. Assume that either $\alg$ is associative or $\alpha \in \erre$. Then we have:
\begin{equation} \label{P:Calpha - AC = I}
\C_\alpha(\A)\alpha-\A\C_\alpha(\A) = \Id.
\end{equation}
\end{Lem}
\begin{proof}
It holds:
\begin{align}
  \Id & = (\A^2 - 2\re(\alpha) \A + |\alpha|^2\Id) \Q_\alpha(\A)  \notag \\
    & = \A^2\Q_\alpha(\A) - \A\Q_\alpha(\A)2\re(\alpha) + \Q_\alpha(\A)|\alpha|^2 \notag \\
    & = \A^2\Q_\alpha(\A) - \A\Q_\alpha(\A)(\alpha^c+\alpha) + \Q_\alpha(\A)\alpha^c\alpha \notag \\
    & = \A(\A\Q_\alpha(\A) - \Q_\alpha(\A)\alpha^c) - (\A\Q_\alpha(\A) - \Q_\alpha(\A)\alpha^c)\alpha \notag \\
    & = -\A\C_\alpha(\A) + \C_\alpha(\A)\alpha, \notag
\end{align}
as desired.
\end{proof}

\begin{Cor}
Let $\A : D(\A) \funzione X$ be a closed right linear operator with $D(\A)$ dense in $X$. Then 
$\rho(\A) \cap \erre = \rho_\s(\A) \cap \erre$ and 
\begin{equation}
\C_\lambda(\A) = \R_\lambda(\A)  
\qquad \forall \lambda \in \rho_\s(\A) \cap \erre.
\end{equation}
In particular $\C_\lambda(\A) \in \Lin^\r(X)$ for $\lambda \in \rho_\s(\A) \cap \erre$. 
\end{Cor}
\begin{proof}
We begin by observing that $\Delta_\lambda(\A) = (\lambda \Id - \A)^2$ whenever $\lambda \in \erre$. Let $\lambda \in \rho_s(\A) \cap \erre$. Since $Q_\lambda(\A)$ commutes with $\A$ (see \eqref{Q e A commutano}) and with $\lambda\Id$ (because $\lambda \in \erre$), we have that
\[
\big((\lambda \Id - \A) \Q_\lambda(\A)\big) (\lambda \Id - \A)=(\lambda \Id - \A)\big((\lambda \Id - \A) \Q_\lambda(\A)\big)=\Id.
\]
Hence $\lambda \in \rho(\A)$ and
\[
\R_\lambda(\A) = (\lambda \Id - \A)^{-1} = (\lambda \Id - \A)\Q_\lambda(\A) = \lambda \Q_\lambda(\A) - \A \Q_\lambda(\A)=\C_\lambda(\A),
\]
where the last equality holds by virtue of  \eqref{P:Calpha - AC = I}. On the other hand, if $\lambda \in \rho(\A) \cap \erre$, then 
$(\lambda \Id - \A)^{-1}$ exists and is continuous by the closed graph theorem. It follows that $\Delta_\lambda(\A)$ is bijective and 
$\Delta_\lambda(\A)^{-1} \in \Lin^\r(X)$, thus $\lambda \in \rho_s(\A)$ and, as in the last formula,
$\C_\lambda(\A) = \R_\lambda(\A)$.
\end{proof}

In the remaining part of this subsection, we assume that
\begin{center}
\emph{$\alg$ is associative.}
\end{center}

Let now consider bounded right linear operators. First, for the sake of completeness, we give the proof of the following lemma, whose quaternionic version was proved in \cite[Theorem 4.2, p. 832]{CoGeSaSt}.

\begin{Lem}\label{L:s-risolvente di B bdd}
If $\B \in \mathscr{L}^\r(X)$, then
\begin{align}\label{Taylor series for the resolvent}
  \C_\alpha(\B) = \sum_{n=0}^\infty \B^n \alpha^{-(n+1)} \qquad \forall \alpha \in Q_\alg, \ |\alpha| > \norma{\B}{}.
\end{align}
Moreover, $\sigma_\s(\B)$ is closed in $Q_\alg$ and 
\begin{equation}\label{spettro sferico nella palla}
  \sigma_\s(\B) \subseteq \{\alpha \in Q_\alg \, : \, |\alpha| \leq \norma{\B}{}\}.
\end{equation}
\end{Lem}
\begin{proof}
Let us note that for $|\alpha| > \norma{\B}{}$ the series converges and
\begin{align}
  \Delta_\alpha(\B) \sum_{n=0}^\infty \B^n \alpha^{-(n+1)} & = \sum_{n=0}^\infty (\B^2 - 2\re(\alpha) \B + |\alpha|^2 \Id)\B^n \alpha^{-(n+1)}  \notag \\
    & = \sum_{n=0}^\infty (\B^{n+2}\alpha^{-(n+1)} - \B^{n+1}(\alpha + \alpha^c)\alpha^{-(n+1)}+\B^n \alpha\alpha^c \alpha^{-(n+1)}) \notag \\
    & = \sum_{n=0}^\infty (\B^{n+2}\alpha^{-(n+1)} - \B^{n+1}\alpha^{-n} + \B^{n+1}\alpha^c\alpha^{-(n+1)} +\B^n \alpha^c \alpha^{-n}) \notag \\
    & = \alpha^c\Id - \B, 
\end{align}
hence, as $D(\B) = X$, from \eqref{risolvente sul dominio} we obtain \eqref{Taylor series for the resolvent}. As a consequence, we infer \eqref{spettro sferico nella palla}. From the continuity of $Q_\alg \lra \Lin^\r(X):\alpha \longmapsto \Delta_\alpha(\B)$ and Corollary \ref{C:A+B invertibile}, it follows that $\rho_\s(\B)$ is open, hence the spherical spectrum of $\B$ is closed.
\end{proof}

\begin{Lem} \label{lem:holomorphic}
Let $\B \in \Lin^\r(X)$ and let $\j \in \su_\alg$. Then the map $\rho_\s(\B) \cap \CC_\j \lra \Lin^\r(X)$, sending $\alpha=a+b\j$ into $\C_\alpha(\B)$
and its partial derivatives $\partial_a\C_\alpha(\B)$ and $\partial_b\C_\alpha(\B)$ 
are continuous and 
satisfy the 
equation
\[
\partial_a\C_\alpha(\B)+\partial_b\C_\alpha(\B)\j=0.
\]
\end{Lem}
\begin{proof}
As $\alpha \longmapsto \Delta_\alpha(\B)$ and the inversion $\A \longmapsto \A^{-1}$ are continuous, the maps 
$\alpha \longmapsto \Q_\alpha(\B)$ and $\alpha \longmapsto \C_\alpha(\B)$ are continuous in $\rho_\s(\B) \cap \ci_\j$. 
By applying standard differential calculus in the real Banach space $_{\erre}\Lin^\r(X)$, we get
\begin{align}
  \partial_a \C_\alpha(\B) 
    & = -\Q_\alpha(\B)(-2\B + 2a\Id) \Q_\alpha(\B) (\alpha^c\Id - \B) + \Q_\alpha(\B) \notag \\
    & = -\Q_\alpha(\B)^{2}(-2\B + 2a\Id) (\alpha^c\Id - \B) + \Q_\alpha(\B)\\
    & = -\Q_\alpha(\B)^2(2\B^2-4a\B+2\B b \j+2a\alpha^c\Id)+ \Q_\alpha(\B), \notag
\end{align}
where we have used commutativity property \eqref{Q e A commutano}. Analogously, we infer
\begin{align}
  \partial_b \C_\alpha(\B) 
    & = -\Q_\alpha(\B) 2b\Id \Q_\alpha(\B) (\alpha^c\Id - \B) - \Q_\alpha(\B)\j \notag \\
    & = -\Q_\alpha(\B)^{2} 2b\Id  (\alpha^c\Id - \B) - \Q_\alpha(\B)\j\\
    & = -\Q_\alpha(\B)^{2}(-2b\B+2b\alpha^c\Id)- \Q_\alpha(\B)\j. \notag
\end{align}
Therefore 
the partial derivatives $\partial_a \C_\alpha(\B)$ and $\partial_b \C_\alpha(\B)$ are continuous, and 
we have
\begin{align}
\partial_a \C_\alpha(\B)+\partial_b \C_\alpha(\B)\j &=-2\Q_\alpha(\B)^{2}(\B^2-2a\B+a\alpha^c\Id+b\alpha^c\j\Id)+2\Q_\alpha(\B)\\
&= -2\Q_\alpha(\B)^2\Delta_\alpha(\B)+2\Q_\alpha(\B)=0, \notag
\end{align}
as desired.
\end{proof}

\begin{Prop} \label{P:indipendenza dell'integrale dalla curva}
Let $\B \in \Lin^\r(X)$, let $f:\OO_D \lra \alg$ be a real slice regular function with $\OO_D \cap \rho_\s(\B) \neq \emptyset$, let $\j \in \su_\alg$ and let $U$ be a bounded open subset of $\CC_\j$ whose boundary $\partial U$ is piecewise $C^1$ and whose closure is contained in $\OO_D \cap \rho_\s(\B) \cap \CC_\j$. Then we have that
\begin{equation}\label{indipendenza dell'integrale dalla curva}
\frac{1}{2\pi} \int_{\partial U} \C_\alpha(\B) \, \j^{-1} f(\alpha) \de \alpha=0,
\end{equation}
where the exact meaning of the above line integral is clarified in Appendix-Section \ref{S:appendix}.
\end{Prop}
\begin{proof}
By Lemma \ref{lem:holomorphic}, the map $\partial U \lra \Lin^\r(X):\alpha \longmapsto \C_\alpha(\B) \, \j^{-1} f(\alpha)$ is continuous and hence the line integral in \eqref{indipendenza dell'integrale dalla curva} makes sense (see Appendix-Section \ref{S:appendix}). Let $(\Lin^\r(X))_{\ci_\j}$ be the (right) complex Banach space obtained endowing $\Lin^\r(X)$ with the right multiplication by scalars in $\ci_\j$. Take a right linear and continuous functional $L : (\Lin^\r(X))_{\ci_\j} \funzione \ci_\j$ and define the function $g:\rho_\s(\B) \cap \CC_\j \lra \CC_\j$ by setting  $g(\alpha):=\langle L,\C_\alpha(\B)\rangle$. By Lemma \ref{lem:holomorphic}, one has that
\[
  \partial_a g(\alpha) + \partial_b g(\alpha) \j= \langle L,\partial_a \C_\alpha(\B) + \partial_b \C_\alpha(\B) \j\rangle = 0.
\]
In other words, $g$ is holomorphic with respect to the complex structure on $\CC_\j$ induced by the multiplication by $\j$. Thanks to the definition of real slice regular function, the same is true for the function $\alpha \longmapsto g(\alpha) \j^{-1} f(\alpha)$ defined on $\OO_D \cap \rho_\s(\B) \cap \CC_\j$. In this way, the classical Cauchy formula implies that
\[
\left\langle L,\frac{1}{2\pi}\int_{\partial U} \C_\alpha(\B) \, \j^{-1}f(\alpha) \de\alpha \right\rangle = \frac{1}{2\pi}\int_{\partial U} g(\alpha) \, \j^{-1} f(\alpha) \de\alpha=0.
\]
Now the result follows from the complex Hahn-Banach theorem applied in $(\Lin^r(X))_{\ci_\j}$. 
\end{proof}

\begin{Rem}
If $X \neq \{0\}$ and $\B \in \Lin^\r(X)$, then $\sigma_\s(\B) \neq \emptyset$. Let us prove this assertion. Let $\j \in \su_\alg$ and let $(\Lin^\r(X))_{\ci_\j}$, $L : (\Lin^\r(X))_{\ci_\j} \funzione \ci_\j$ and $g:\rho_\s(\B) \cap \CC_\j \lra \CC_\j$ be as in the proof of the preceding proposition. Suppose that $\sigma_\s(\B)=\emptyset$. Then $g$ is an entire holomorphic function on $\CC_\j$. Moreover, by \eqref{Taylor series for the resolvent}, we have
\[
|g(\alpha)| \leq M \|L\|_\j \, |\alpha|^{-1} \qquad \forall \alpha \in \CC_\j, \, |\alpha| \geq 1+\|B\|,
\]
where $\|L\|_\j$ is the norm of $L$ as continuous functional on $(\Lin^\r(X))_{\ci_\j}$ and $M=\sum_{n \geq 0}\|B\|^n(1+\|B\|)^{-n}$. By Liouville's theorem, $g$ is the null function. Thanks to the complex Hahn-Banach theorem, we infer that $\C_\alpha(\B)=0$ for every $\alpha \in \CC_\j$. This is impossible by the definition of spherical resolvent operator.
\end{Rem}

\begin{Lem}\label{L:identity as integral of the resolvent}
Let $\B \in \Lin^\r(X)$, let $\j \in \su_\alg$ and let $r \in \opint{0,\infty}$ such that $\sigma_\s(\B) \cap \CC_\j \subseteq B_\j(r)$, where $B_\j(r):=\{\alpha \in \CC_\j \, : \, |\alpha| <r\}$. Then we have
\begin{equation}\label{identity as integral of the resolvent}
 \Id = \frac{1}{2\pi} \int_{\partial B_\j(r)} \C_\alpha(\B) \, \j^{-1} \de \alpha. 
\end{equation}
\end{Lem}
\begin{proof}
Thanks to Proposition \ref{P:indipendenza dell'integrale dalla curva} applied with $f$ constantly equal to $1$, we can suppose that $r > \norma{\B}{}$. By the classical Cauchy formula, and by the fact that $\B$ is right linear and continuous, we have
\[
  0 = \B^n \frac{1}{2\pi} \int_{\partial B_\j(r)} \alpha^{-(n+1)} \j^{-1} \de \alpha =
  \frac{1}{2\pi} \int_{\partial B_\j(r)} \B^n \alpha^{-(n+1)} \j^{-1} \de \alpha \qquad \forall n \ge 1
\]
and
\[
  \Id = \Id \frac{1}{2\pi} \int_{\partial B_\j(r)} \alpha^{-1} \j^{-1} \de \alpha =
  \frac{1}{2\pi} \int_{\partial B_\j(r)} \Id \alpha^{-1} \j^{-1} \de \alpha.
\]
Hence, bearing in mind Lemma \ref{L:s-risolvente di B bdd}, we infer that
\begin{align*}
\Id &= \sum_{n=0}^\infty \frac{1}{2\pi} \int_{\partial B_\j(r)} \B^n \alpha^{-(n+1)}\j^{-1} \de \alpha=\frac{1}{2\pi} \int_{\partial B_\j(r)} \sum_{n=0}^\infty \B^n \alpha^{-(n+1)}\j^{-1} \de \alpha \\ &=\frac{1}{2\pi} \int_{B_\j(r)} \C_\alpha(\B) \, \j^{-1} \de \alpha,
\end{align*}
as desired.
\end{proof}

The next lemma provides a bridge between bounded and unbounded right linear operators and is based on a technique used in \cite[Lemma VII.9.2, p. 600]{DuSc} and also in \cite[Definition 3.10]{CoSa10}.

\begin{Lem}\label{L:lemma su Phi}
Let $\A : D(\A) \funzione X$ be a closed right linear operator with $D(\A)$ dense in $X$. Assume that $\opint{0,\infty} \subseteq \rho_\s(\A)$ and take $\lambda \in \opint{0,\infty}$. Define $\Phi : Q_\alg \setminus \{\lambda\} \lra Q_\alg$ and $\B \in \mathscr{L}^\r(X)$ by setting
\begin{equation}\label{Phi(alpha)}
\Phi(\alpha):=(\alpha - \lambda)^{-1}
\end{equation}
and
\begin{equation}\label{def di B}
  \B := -\C_\lambda(\A) = -\R_\lambda(\A) = (\A - \lambda \Id)^{-1}.
\end{equation} 
Then we have:
\begin{equation}\label{relation for rho(A) and rho(B)}
  \Phi(\sigma_\s(\A)) = \sigma_\s(\B)
\end{equation}
and
\begin{equation}\label{relation for R(A) and R(B)}
  \C_\alpha(\A) = -\B \C_{\Phi(\alpha)}(\B)\Phi(\alpha) \qquad \forall \alpha \in Q_\alg \setminus \{\lambda\}.
\end{equation}
\end{Lem}

\begin{proof}
Let $\alpha \in Q_\alg \setminus \{\lambda\}$ and let $\beta:=\Phi(\alpha)$. As $\alpha = 1/\beta + \lambda$, we have 
\begin{equation}
  \re(\alpha) = \re(\beta)/|\beta|^2 + \lambda, 
\end{equation}
hence, since $\A = \B^{-1} + \lambda \Id$, we find 
\begin{align}
  \Delta_\alpha(\A) 
    &  = \A^2 - 2\re(\alpha) \A + |\alpha|^2\Id \notag \\
    & = (\B^{-1} + \lambda\Id)^2 - 2\re(\alpha)(\B^{-1} + \lambda\Id) + |\alpha|^2\Id \notag \\
    & = \B^{-2} + 2(\lambda - \re(\alpha)) \B^{-1} + (|\lambda|^2 - 2\lambda \re(\alpha) + |\alpha|^2)\Id \notag \\
    & = \B^{-2}\left(\Id + 2(\lambda - \re(\alpha)) \B + (\alpha- \lambda)(\alpha^c - \lambda)\B^2\right) \notag \\
    & = \B^{-2}\left(|\beta|^2\Id + 2\re(\beta)/|\beta|^2 \B + |\beta|^{-2}\B^2\right) \notag \\
    & = |\beta|^{-2} \B^{-2} \Delta_\beta(\B). 
\end{align}
It follows that $\Delta_\beta(B)$ is invertible if and only if $\Delta_\alpha(\A)$ is invertible. Moreover,
\begin{equation}\label{relazione fra Q(A) e Q(B)}
  \Q_\alpha(\A) = |\beta|^2 \Q_\beta(\B) \B^2 = |\beta|^{2} \B^2 \Q_\beta(\B). 
\end{equation}
We infer at once that $\Q_\beta(\B)$ is continuous if and only if $Q_\alpha(\A)$ is continuous. This proves \eqref{relation for rho(A) and rho(B)}. Finally, equation \eqref{relation for R(A) and R(B)} is proved by means of \eqref{relazione fra Q(A) e Q(B)}, indeed
\begin{align}
  \C_\alpha(\A) 
    & = \Q_\alpha(\A)\alpha^c - \A\Q_\alpha(\A) \notag \\
    & = \Q_\alpha(\A)\left(\beta/|\beta|^2 + \lambda\right) - \A\Q_\alpha(\A) \notag \\
    & = |\beta|^{2} \B^2 \Q_\beta(\B) \left(\beta/|\beta|^2 + \lambda\right) - \A |\beta|^{2} \B^2 \Q_\beta(\B) \notag \\
    & = \B^2 \Q_\beta(\B)(\beta + |\beta|^2\lambda) - \A \B^2 \Q_\beta(\B) |\beta|^{2}\notag \\
    & = \B^2 \Q_\beta(\B)(\beta + |\beta|^2\lambda) - (\A-\lambda \Id) \B^2 \Q_\beta(\B)|\beta|^{2} - 
           \lambda \B^2 \Q_\beta(\B) |\beta|^{2} \notag \\
    & = \B^2 \Q_\beta(\B)(\beta + |\beta|^2\lambda) - \B \Q_\beta(\B)|\beta|^{2} -  \B^2 \Q_\beta(\B) \lambda |\beta|^{2} \notag \\    
    & = \B^2 \Q_\beta(\B)\beta - \B \Q_\beta(\B)|\beta|^{2} \notag \\
    & = \B\left( \B\Q_\beta(\B) - \Q_\beta(\B) \overline{\beta}\right)\beta \notag \\ 
    & = -\B\C_\beta(\B)\beta. \notag  
\end{align}
\end{proof}

By combining Lemma \ref{L:lemma su Phi} with Lemma \ref{lem:holomorphic} and Proposition \ref{P:indipendenza dell'integrale dalla curva}, we obtain at once:

\begin{Cor} \label{cor:inv-int-A}
Let $\A : D(\A) \funzione X$ be a closed right linear operator with $D(\A)$ dense in $X$ and let $\j \in \su_\alg$. Assume that $\opint{0,\infty} \subseteq \rho_\s(\A)$. Then the map $\rho_\s(\A) \cap \CC_\j \lra \Lin^\r(X):\alpha \longmapsto \C_\alpha(\A)$ 
and its partial derivatives $\partial_a\C_\alpha(\A)$ and $\partial_b\C_\alpha(\A)$ are continuous and satisfy the 
equation
\[
\partial_a\C_\alpha(\A)+\partial_b\C_\alpha(\A)\j=0.
\]

In particular, if $f:\OO_D \lra \alg$ is a real slice regular function with $\OO_D \cap \rho_\s(\A) \neq \emptyset$ and if $U$ is a piecewise $C^1$ bounded open subset of $\CC_\j$ whose closure is contained in $\OO_D \cap \rho_\s(\A) \cap \CC_\j$, then we have that
\begin{equation}\label{inv.int-A}
\frac{1}{2\pi} \int_{\partial U} \C_\alpha(\A) \, \j^{-1} f(\alpha) \de \alpha=0.
\end{equation}
\end{Cor}


\section{Generation theorems for semigroups in real or complex Banach spaces}\label{S:classical generation theorems}

In this section, for later use, we recall the basic results of the theory of operator semigroups in a real or complex Banach space. For the proofs, we refer e.g. to \cite{EnNa, Pa} and we stress the fact that in these monographs the functional calculus is not used.
We assume that
\begin{equation}\label{Y C-Banach space}
  \text{\emph{$Y$ is a real or complex Banach space with norm $\norma{\cdot}{}$.}}
\end{equation}

\begin{Def}
A mapping $\T : \clsxint{0,\infty} \funzione \Lin(Y)$ is called \emph{(operator) semigroup} if
\begin{align}
  & \T(t+s) = \T(t)\T(s) \qquad \forall t, s > 0, \\
  & \T(0) = \Id.
\end{align}
We also say that $\T$ is a \emph{semigroup in $Y$}. A semigroup $\T$ 
in $Y$ 
is called \emph{uniformly continuous} if $\T \in C(\clsxint{0,\infty};\Lin(Y))$. A semigroup $\T$ is called \emph{strongly continuous} if $\T(\cdot)y \in C(\clsxint{0,\infty};Y)$ for every $y \in Y$.
\end{Def}

The case of uniformly continuous semigroups can be treated by means of the operator-valued exponential, and is described by the following theorem, whose proof can be found, e.g., in \cite[Section 1.1]{Pa}.

\begin{Thm}\label{T:generation thm for unif cont semigroups}
The following assertions hold.
\begin{itemize}
\item[$(\mr{a})$]
  If $\T : \clsxint{0,\infty} \funzione \Lin(Y)$ is a uniformly continuous semigroup then $\T$ is differentiable and, setting $\A := \T'(0) \in \Lin(Y)$, one has
\begin{align}
  & \T'(t) = \A\T(t) \qquad \forall t \ge 0, \label{DE-unif. cont. case} \\
  & \T(0) = \Id.
\end{align}
Moreover, $\T(t) = e^{t\A}$ for all $t \ge 0$.
\item[$(\mr{b})$]
  If $\A \in \Lin(Y)$, the map $\T : \clsxint{0,\infty} \funzione \Lin(Y) : t \longmapsto e^{t\A}$ is a uniformly continuous semigroup and it is the unique solution of the Cauchy problem
\begin{align}
  & \T'(t) = \A \T(t) \qquad \forall t \ge 0, \\
  & \T(0) = \Id.
\end{align}  
\end{itemize}
\end{Thm}

Uniform continuity is a too strong condition for applications, but uniformly continuous semigroups are needed in order to deal with the most natural case of strongly continuous semigroups. We start by recalling the notion of generator. 

\begin{Def}
Let $\T : \clsxint{0,\infty} \funzione \Lin(Y)$ be a strongly continuous semigroup. The \emph{generator of $\T$} is the linear operator 
$\A : D(\A) \funzione Y$ defined by
\begin{align}
  & D(\A) := \big\{y \in Y \, : \, \exists \lim_{h \to 0} (1/h)(\T(h)y - y) = (\de/\de t) \T(t)y\big|_{t=0}\big\}, \\
  & \A y := \lim_{h \to 0}\frac{1}{h}(\T(h)y - y), \qquad \forall y \in D(\A).
\end{align}
\end{Def}

Now we can state the so-called \emph{Generation Theorem}.

\begin{Thm}[Feller, Miyadera, Phillips]\label{T:generation thm}
The following assertions hold.
\begin{itemize}
\item[$(\mr{a})$]
Let $\A:D(\A) \lra Y$ be a linear operator with $D(\A)$  dense in $Y$. Assume that there are constants $M \in \opint{1,\infty}$ and $w \in \erre$ such that $\opint{w,\infty} \subseteq \rho(\A)$ and
  \begin{equation}\label{resolvent cond. for real semigroups}
    \norma{\R_\lambda(\A)^n}{} \le \frac{M}{(\lambda - w)^{n}} \qquad \forall n \in \enne, \quad \forall \lambda > w.
  \end{equation}
  Then $\A$ is the generator of the strongly continuous semigroup $\T : \clsxint{0,\infty} \funzione \Lin(Y)$ defined by 
  \begin{equation}
     \T(t)y = \lim_{n \to \infty} e^{t\A_{n}}y, \quad y \in Y, \qquad
     \text{where } \A_{n} := n\A\R_n(\A) 
     \in \Lin(Y).
  \end{equation}
  Moreover, $\norma{\T(t)}{} \le Me^{wt}$ for all $t \ge 0$.
\item[$(\mr{b})$]
  Let $\T : \clsxint{0,\infty} \funzione \Lin(Y)$ be a strongly continuous semigroup. Then there are constants $M \in \opint{1,\infty}$ and $w \in \erre$ such that $\norma{\T(t)}{} \le Me^{tw}$ for every $t \ge 0$. Moreover, the generator $\A$ of $\T$ is closed, $D(\A)$ is dense, 
  $\opint{w,\infty} \subseteq \rho(A)$ and 
  $\norma{\R_\lambda(\A)^n}{} \le \frac{M}{(\lambda - w)^{n}}$ for all $n \in \enne$ and for all $\lambda > w$.
\end{itemize}
In both cases $(\mr{a})$ and $(\mr{b})$, we have that
\begin{equation}\label{R = trasf laplace}
  \R_\lambda(\A)y = \int_{0}^{\infty}e^{-t\lambda} \T(t)y \de t \qquad \forall \lambda > w, \quad \forall y \in Y.
\end{equation}
\end{Thm}

For the proof of Theorem \ref{T:generation thm}, we refer to \cite[Section 1.5]{Pa}, where it is shown that the proof can be reduced to the important particular case in which $M = 1$ and $w = 0$. In such a situation, $\T$ is called a strongly continuous \emph{contraction} semigroup, and, in order to prove that $\A$ generates a semigroup, it is enough to verify \eqref{resolvent cond. for real semigroups} for $n = 1$ only. Indeed, we have the following

\begin{Thm}[Hille, Yosida]\label{T:generation thm, contraction case}
The following assertions hold.
\begin{itemize}
\item[$(\mr{a})$]
  Let $D(\A)$ be a dense vector subspace of $Y$ and let $\A:D(\A) \funzione \mathrm{Y}$ be a linear operator. Assume that $\opint{0,\infty} \subseteq \rho(\A)$ and that 
  \begin{equation}
    \norma{\R_\lambda(\A)}{} \le \frac{1}{\lambda} \qquad \forall \lambda > 0.
  \end{equation}
  Then $\A$ is the generator of a strongly continuous semigroup $\T : \clsxint{0,\infty} \funzione \Lin(Y)$ such that $\norma{\T(t)}{} \le 1$ for all $t \ge 0$.
\item[$(\mr{b})$]
  Let $\T : \clsxint{0,\infty} \funzione \Lin(Y)$ be a strongly continuous semigroup
  such that $\norma{\T(t)}{} \le 1$ for every $t \ge 0$. Then the generator $\A$ of $\T$ is closed, $D(\A)$ is dense, 
  $\opint{0,\infty} \subseteq \rho(A)$ and 
  $\norma{\R_\lambda(\A)}{} \le \frac{1}{\lambda}$ for all $\lambda > 0$.
\end{itemize}
\end{Thm}

For the proof we refer again to \cite[Section 1.3]{Pa}. Finally, we state the result relating strongly continuous semigroups with the Cauchy problem
(see, e.g., \cite[Section 4.1]{Pa}).

\begin{Thm}\label{T:semigroups and diff eq}
Let $\A : D(\A) \funzione Y$ be a closed linear operator with  $D(\A)$ dense in $Y$. Then the following conditions are equivalent:
\begin{itemize}
\item[$(\mr{a})$]
  $\A$ is the generator of a strongly continuous semigroup $\T$.
\item[$(\mr{b})$]
  $\rho(\A) \cap \erre \neq \vuoto$ and, for every $y \in D(\A)$, there exists a unique $u \in C^1(\clsxint{0,\infty};Y)$ such that $u(t) \in D(\A)$ for all $t \geq 0$ and
  \begin{align}
    & u'(t) = \A u(t) \qquad \forall t \ge 0, \label{eq. diff. per A str. cont.}\\
    & u(0) = y. \label{in. cond. per A str. cont.}
  \end{align}
\end{itemize} 
In such cases, the unique solution of \eqref{eq. diff. per A str. cont.}--\eqref{in. cond. per A str. cont.} is $u(t) = \T(t)y$.
\end{Thm}


\section{Generation theorems for semigroups in Banach $\alg$-bimodules}\label{S:str cont semigroups in A-modules}

In this section, we generalize the results contained in \cite{CoSa} by proving the generation theorems for semigroups
in Banach $\alg$-bimodules. Our proofs are different and simpler, since we reduce to the classical case thanks to Lemma \ref{L:L^r chiuso in L}.

We recall that $\alg$ is a real algebra satisfying \eqref{eq:assumption-1} and \eqref{eq:assumption-2}. We assume that
\[
  \text{\emph{$X$ is a Banach $\alg$-bimodule with norm $\norma{\cdot}{}$.}}
\]

\begin{Def}
A mapping $\T : \clsxint{0,\infty} \funzione \Lin^\r(X)$ is called \emph{(right linear operator) semigroup} if
\begin{align}
  & \T(t+s) = \T(t)\T(s) \qquad \forall t, s > 0, \\
  & \T(0) = \Id.
\end{align}
We also say that $\T$ is a \emph{(right linear operator) semigroup in $X$}. A semigroup $\T$ is called \emph{uniformly continuous} if $\T \in C(\clsxint{0,\infty};\Lin^\r(X))$. A semigroup $\T$ is called \emph{strongly continuous} if 
$\T(\cdot)x \in C(\clsxint{0,\infty};X)$ for every $x \in X$.
\end{Def}

\begin{Rem}
By Lemma \ref{L:continuita invariante}, we could also say that $\T$ is a uniformly continuous (resp. strongly continuous) semigroup in $X$ if $\T$ is $\Lin^\r(X)$-valued and $\T$ is a uniformly continuous (resp. strongly continuous) semigroup in 
$_\erre X$.
\end{Rem}

\begin{Thm}\label{T:A-generation thm for unif cont semigroups}
The following assertions hold.
\begin{itemize}
\item[$(\mr{a})$]
If $\T : \clsxint{0,\infty} \funzione \Lin^\r(X)$ is a uniformly continuous semigroup then $\T$ is differentiable and, setting 
  $\A := \T'(0) \in \Lin^\r(X)$, one has
\begin{align}
  & \T'(t) = \A\T(t) \qquad \forall t \ge 0, \label{A-DE-unif. cont. case} \\
  & \T(0) = \Id.
\end{align}
Moreover $\T(t) = e^{t\A}$ for all $t \ge 0$.
\item[$(\mr{b})$]
  If $\A \in \Lin^\r(X)$, the map $\T : \clsxint{0,\infty} \funzione \Lin^\r(X) : t \longmapsto e^{t\A}$ is a uniformly continuous semigroup and it is the unique solution of the Cauchy problem
\begin{align}
  & \T'(t) = \A \T(t) \qquad \forall t \ge 0, \\
  & \T(0) = \Id.
\end{align}  
\end{itemize}
\end{Thm}

\begin{proof}
By Lemma \ref{L:continuita invariante}, we have that $\T$ is also a uniformly continuous semigroup in $_{\erre}X$, therefore we can apply Theorem \ref{T:generation thm for unif cont semigroups}$(\mr{a})$. We conclude part (a) by invoking Lemma \ref{L:L^r chiuso in L}, which implies that $\A = \T'(0) \in \Lin^\r(X)$. Concerning $(\mr{b})$, we can apply 
Theorem \ref{T:generation thm for unif cont semigroups}$(\mr{b})$ to $\A \in \Lin^\r(X) \subseteq \Lin(_\erre X)$ and use again Lemma 
\ref{L:L^r chiuso in L} in order to obtain that $e^{t\A}$ belongs to $\Lin^\r(X)$ for every $t \ge 0$.
\end{proof}

The generator of a strongly continuous semigroup is defined as in the classical case.

\begin{Def}
Let $\T : \clsxint{0,\infty} \funzione \Lin^\r(X)$ be a strongly continuous semigroup. The \emph{generator of $\T$} is the right linear operator $\A : D(\A) \funzione X$ defined by
\begin{align}
  & D(\A) := \big\{x \in X\ :\ \exists \lim_{h \to 0} (1/h)(\T(h)x - x) = (\de/\de t) \T(t)x\big|_{t=0}\big\}, \\
  & \A x := \lim_{h \to 0}\frac{1}{h}(\T(h)x - x), \qquad x \in D(\A).
\end{align}
\end{Def}

We can now prove the generation theorem for semigroups in a Banach $\alg$-bimodule.

\begin{Thm}\label{T:A-generation thm}
The following assertions hold.
\begin{itemize}
\item[$(\mr{a})$]
Let $\A : D(\A) \funzione X$ be a closed right linear operator with $D(\A)$ dense in $X$. Assume that there are constants $M \in \opint{1,\infty}$ and $w \in \erre$ such that $\opint{w,\infty} \subseteq \rho_\s(\A)$ and
  \begin{equation}\label{cond on R^n-A}
    \norma{\C_\lambda(\A)^n}{} \le \frac{M}{(\lambda - w)^{n}} \qquad \forall n \in \enne, \quad \forall \lambda > w.
  \end{equation}
  Then $\A$ is the generator of the strongly continuous semigroup $\T : \clsxint{0,\infty} \funzione \Lin^\r(X)$ defined by 
  \begin{equation}
     \T(t)x = \lim_{n \to \infty} e^{t\A_{n}}x, \quad x \in X, \qquad \text{where } \A_{n} := n\A\C_n(\A) \in \Lin^\r(X).
  \end{equation}
  Moreover, $\norma{\T(t)}{} \le Me^{wt}$ for all $t \ge 0$.
\item[$(\mr{b})$]
 Let $\T : \clsxint{0,\infty} \funzione \Lin^\r(X)$ be a strongly continuous semigroup. Then there are constants $M \in \opint{1,\infty}$ and $w \in \erre$ such that $\norma{\T(t)}{} \le Me^{tw}$ for every $t \ge 0$. Moreover the generator $\A$ of $\T$ is closed, $D(\A)$ is dense, $\opint{w,\infty} \subseteq \rho_\s(A)$ and $\norma{\C_\lambda(\A)^n}{} \le \frac{M}{(\lambda - w)^{n}}$ for all $n \in \enne$ and for all $\lambda > w$.
\end{itemize}
In both cases $(\mr{a})$ and $(\mr{b})$, we have that
\begin{equation}\label{R = trasf laplace-A}
  \C_\lambda(\A)x = \int_{0}^{\infty}e^{-t\lambda} \T(t)x \de t \qquad \forall \lambda > w, \quad \forall x \in X.
\end{equation}
\end{Thm}

\begin{proof}
Observe that $D(\A)$ is a real vector subspace of $_\RR X$ and $\A$ is $\erre$-linear. Moreover, $\A$ is closed and $D(\A)$ is dense in $X$. Since $\C_\lambda(\A) = \R_\lambda(\A)$ for real $\lambda$, thanks to \eqref{cond on R^n-A}, we can apply Theorem \ref{T:generation thm}(a) and infer that $\A$ is the generator of the strongly continuous semigroup 
$\T : \clsxint{0,\infty} \funzione \Lin(_{\erre}X)$ given by 
\begin{equation}
     \T(t) = \lim_{n \to \infty} e^{t\A_{n}} \quad
     \text{where } \A_{n} := n\A\R_n(\A) \in \Lin(_{\erre}X),
  \end{equation}
and $\norma{\T(t)}{X} = \norma{\T(t)}{_\erre X} \le Me^{wt}$ for all $t \geq 0$. As $n\Id - \A$ is right linear, by Lemma \ref{L:inverse is linear}, $\R_n(\A)$ is right linear. Thus $\A_{n}$ and $e^{t\A_{n}}$ are right linear as well. From Lemma \ref{L:L^r chiuso in L}, it follows that $\T(t)$ is right linear and (a) is proved. Concerning (b), as $\T$ is also a strongly continuous semigroup in $\Lin(_{\erre}X)$, we can apply Theorem \ref{T:generation thm}(b) and infer the thesis, indeed $\A \in \Lin^\r(X)$ because of Lemma \ref{L:L^r chiuso in L}. Formula 
\eqref{R = trasf laplace-A} obviously holds in both cases.
\end{proof}

Using similar arguments, Theorems \ref{T:generation thm, contraction case} and \ref{T:semigroups and diff eq} can be rephrased
in the new framework. We limit ourselves to state the relation with differential equations.

\begin{Thm}\label{T:semigroups and diff eq-A}
Let $\A : D(\A) \funzione X$ be a closed right linear operator with  $D(\A)$ dense in $X$. Then the following conditions are equivalent:
\begin{itemize}
\item[$(\mr{a})$]
  $\A$ is the generator of a strongly continuous semigroup $\T$.
\item[$(\mr{b})$]
  $\rho_s(\A) \cap \erre \neq \vuoto$ and, for every $x \in D(\A)$, there exists a unique $u \in C^1(\clsxint{0,\infty};X)$ such that
  \begin{align}
    & u'(t) = \A u(t) \qquad \forall t \ge 0, \label{eq. diff. per A str. cont.-A} \\
    & u(0) = x. \label{in. cond. per A str. cont.-A}
  \end{align}
\end{itemize} 
In such cases, the unique solution of \eqref{eq. diff. per A str. cont.-A}--\eqref{in. cond. per A str. cont.-A} is $u(t) = \T(t)x$.
\end{Thm}


\section{Spherical sectorial operators in Banach $\alg$-bimodules} \label{sec:spherical-sect-operators}

In this section, we study the analogous of the sectorial operators in the new framework of right linear operators acting on Banach $\alg$-bimodules. We will show that these kind of operators can be expressed as a line integral 
involving the exponential function. Therefore, it turns out that the associated semigroup is analytic in time.
We would like to draw the reader's attention that, for the study of these operators, we actually see the functional calculus at work, whereas in our proof of the generation theorems only the real structure of the Banach $\alg$-bimodules plays a role.

We assume that $\alg$ is a real algebra satisfying \eqref{eq:assumption-1} and \eqref{eq:assumption-2},
\begin{center}
\emph{$\alg$ is associative}
\end{center}
and
\begin{center}
\emph{$X$ is a Banach $\alg$-bimodule with norm $\norma{\cdot}{}$.}
\end{center}

\begin{Def}
Define the \emph{argument function} $\arg:Q_\alg \setminus \{0\} \lra \clint{0,\pi}$ on $Q_\alg$ as follows. If $\alpha \in Q_\alg \setminus \RR$, then there exist, and are unique, $\j \in \su_\alg$, $\rho \in \opint{0,\infty}$ and $\theta \in \opint{0,\pi}$ such that $\alpha=\rho e^{\theta\j} \in \mathbb{C}_\j$. Thus, we can define $\arg(\alpha):=\theta$. Moreover we set: $\arg(\alpha):=0$ if $\alpha \in \opint{0,\infty}$ and $\arg(\alpha):=\pi$ if $\alpha \in \opint{-\infty,0}$. 
\end{Def}

\begin{Def}
Let $\A : D(\A) \funzione X$ be a closed right linear operator with  $D(\A)$ dense in $X$ and let $\delta \in \opint{0,\pi/2}$. We say that $\A$ is a \emph{spherical $\delta$-sectorial operator} if 
\begin{equation}
\Sigma_\delta:=\big\{\alpha \in Q_\alg \setminus \{0\} \, : \, \arg(\alpha)<\pi/2 + \delta\big\} \subseteq \rho_\s(\A).
\end{equation}
If $\A$ is a spherical $\delta$-sectorial operator for some $\delta \in \opint{0,\pi/2}$, then we say that $\A$ is a \emph{spherical sectorial operator}.
\end{Def}

Let $\j \in \su_\alg$ and let $\Omega$ be an open subset of $\CC_\j$. By the term \textit{$C^1$-path of $\Omega$}, 
we mean a $C^1$-map $\gamma:\clint{a,b} \lra \Omega$ such that either $\gamma$ is injective or it is not, but $\gamma(t)=\gamma(s)$ only if $(t,s)=(a,b)$ or $(t,s)=(b,a)$. 
Denote by $-\gamma:\clint{a,b} \lra \Omega$ 
the $C^1$-path of $V$ obtained changing the orientation of $\gamma$; namely, $-\gamma(t):=\gamma(a+b-t)$. Given a continuous map $g:U \lra \CC_\j$, we define $\int_\gamma g(\alpha)|\de \alpha|:=\int_a^b g(\gamma(t))|\gamma'(t)| \de t$. Let 
$\gamma_1,\ldots,\gamma_n$ be $C^1$-paths of $\Omega$ 
and let $\Gamma=\gamma_1+\ldots+\gamma_n$ be the formal sum of the $\gamma_i$'s. As usual, we set $\int_\Gamma g(\alpha) \de \alpha:=\sum_{i=1}^n\int_{\gamma_i}g(\alpha) \de \alpha$, where $\int_{\gamma_i}g(\alpha) \de \alpha$ is the line integral of $g$ along $\gamma_i$. If $g$ is injective and of class $C^1$, then we define $g \circ \Gamma:=g \circ \gamma_1+\ldots+g \circ \gamma_n$. Let $\gamma'_1,\ldots,\gamma'_m$ be other $C^1$-paths of $\Omega$ and let $\Gamma':=\gamma'_1+\ldots+\gamma'_m$. In this situation, we denote by $\Gamma-\Gamma'=\gamma_1+\ldots+\gamma_n-\gamma'_1-\ldots-\gamma'_m$ the formal sum $\gamma_1+\ldots+\gamma_n+(-\gamma'_1)+\ldots+(-\gamma'_m)$ of $C^1$-paths of $\Omega$.

\begin{Def}\label{D:curva}
Let $\j \in \su_\alg$, let $r,s \in \opint{0,\infty}$ with $r<s$ and let $\eta,\xi \in \opint{0,2\pi}$ 
with $\eta<\xi$. Define the $C^1$-paths $\ray(\j \, ;r,s;\eta)$ and $\ce(\j \, ;r;\eta,\xi)$ of $\CC_\j$ by setting 
\begin{align}
& \ray(\j \, ;r,s;\eta)(t):=te^{\eta\j}, \qquad t \in \clint{r,s},\\
& \ce(\j \, ;r;\eta,\xi)(t):=re^{t\j}, \qquad t \in \clint{\eta,\xi}
\end{align}
and the $C^1$-maps $\ray^-(\j \, ;r;\eta):\cldxint{-\infty,-r} \lra \CC_\j$ and $\ray^+(\j \, ;r;\eta):\cldxint{r,\infty} \lra \CC_\j$ by setting
\begin{align}
& \ray^-(\j \, ;r;\eta)(t):=-te^{\eta\j}, \qquad t \in \cldxint{-\infty,-r},\\
& \ray^+(\j \, ;r;\eta)(t):=te^{\eta\j}, \qquad t \in \clsxint{r,\infty}.
\end{align}
Moreover, we denote by $\Gamma(\j \, ;r,s;\eta)$ and $\Gamma(\j \, ;r;\eta)$ the following formal sums of $C^1$-maps:
\begin{align}
&\Gamma(\j \, ;r,s;\eta):=-\ray(\j \, ;r,s;-\eta)+\ce(\j \, ;r;-\eta,\eta)+\ray(\j \, ;r,s;\eta),\\
&\Gamma(\j \, ;r;\eta):=\ray^-(\j \, ;r;-\eta)+\ce(\j \, ;r;-\eta,\eta)+\ray^+(\j \, ;r;\eta). \label{curva gamma}
\end{align}

Let $U$ be an open neighborhood in $\CC_\j$ of the union of the images of $\ray^-(\j \, ;r;-\eta)$, of $\ce(\j \, ;r;-\eta,\eta)$ and of 
$\ray^+(\j \, ;r;\eta)$, and let $\Psi:U \lra \Lin^\r(X)$ be a continuous map. Define
\begin{align}
& \int_{\ray^-(\j \, ;r;-\eta)}\Psi(\alpha) \de\alpha:=\lim_{h \rightarrow \infty}\int_{-\ray(\j \, ;r,h;-\eta)}\Psi(\alpha) \de\alpha \label{eq:-},\\
& \int_{\ray^+(\j \, ;r;\eta)}\Psi(\alpha) \de\alpha:=\lim_{h \rightarrow \infty}\int_{\ray(\j \, ;r,h;\eta)}\Psi(\alpha) \de\alpha \label{eq:+},
\end{align}
if such limits exist in $\Lin^\r(X)$. If both limits exist, we say that the integral $\int_{\Gamma(\j \, ;r;\eta)}\Psi(\alpha) \de\alpha$ converges in $\Lin^\r(X)$ and we set
\[
\int_{\Gamma(\j \, ;r;\eta)}\Psi(\alpha) \de\alpha:=
\int_{\ray^-(\j \, ;r;-\eta)}\Psi(\alpha) \de\alpha+
\int_{\ce(\j \, ;r;-\eta,\eta)}\Psi(\alpha) \de\alpha+
\int_{\ray^+(\j \, ;r;\eta)}\Psi(\alpha) \de\alpha.
\]
\end{Def}

\noindent In this situation, we have
\begin{equation} \label{eq:limit}
\int_{\Gamma(\j \, ;r;\eta)}\Psi(\alpha) \de\alpha=\lim_{h \rightarrow \infty}\int_{\Gamma(\j \, ;r,h;\eta)}\Psi(\alpha) \de\alpha.
\end{equation}

\begin{Lem}\label{L:indipendenza di U(t) dalla curva}
Let $D(\A)$ be a dense right $\alg$-submodule of $X$, let $\A : D(\A) \lra X$ be a spherical $\delta$-sectorial operator for some $\delta \in \opint{0,\pi/2}$, let $r \in \opint{0,\infty}$ and let $\eta \in \opint{\pi/2,\pi/2+\delta}$. Assume that there exists $M \in \opint{0,\infty}$ such that
\begin{equation}\label{stima per risolvente di A settoriale1}
\norma{\C_\alpha(\A)}{} \le \frac{M}{|\alpha|} \qquad \forall \alpha \in \Sigma_\delta.
\end{equation}
Then, for every $t>0$, the integral
\begin{equation}\label{espressione integrale per il risolvente1}
  \T(t):= \frac{1}{2\pi} \int_{\Gamma(\j \, ;r;\eta)} \C_\alpha(\A) \, \j^{-1} e^{t\alpha} \de \alpha
\end{equation}
converges in $\Lin^\r(X)$ and does not depend on the choices of $r$ and of $\eta$. Moreover, the corresponding map $\T:\opint{0,\infty} \lra \Lin^r(X)$ is differentiable and we have:
\[
\text{$\sup_{t>0}\norma{\T(t)}{} < \infty \;$ and $\; \frac{\de\ }{\de t}\T(t)x = \A \T(t)x \quad \forall t>0$.}
\]
\end{Lem}
\begin{proof}
Fix $t \in \opint{0,\infty}$. First, we observe that
\begin{align}
\T(t)&=\frac{1}{2\pi} \int_{-\eta}^\eta \C_{r e^{\theta \j}}(\A) r e^{\theta \j} e^{tr e^{\theta \j}}\de \theta \notag\\
&\sp-\frac{1}{2\pi}\int_{-\infty}^{-r} \C_{-\rho e^{-\eta \j}}(\A) \, \j^{-1} e^{-\eta \j} e^{-t\rho e^{-\eta \j}}\de \rho+\frac{1}{2\pi} \int_r^{\infty} \C_{\rho e^{\eta \j}}(\A) \, \j^{-1} e^{\eta \j}e^{t\rho e^{\eta \j}}\de \rho. \label{eq:|T|}
\end{align}
By using \eqref{stima per risolvente di A settoriale1}, we obtain that
\begin{align}\label{first estimate for U}
\norma{\T(t)}{} \le \frac{M}{2\pi} \int_{-\eta}^\eta e^{t r \cos(\theta)} \de \theta + \frac{M}{\pi} \int_r^\infty \frac{e^{t \rho \cos(\eta)}}{\rho} \de \rho.
\end{align}
In particular, it follows that the last two integrals of \eqref{eq:|T|} are absolutely convergent. 

Let us prove that the integral in \eqref{espressione integrale per il risolvente1} does not depend on $r$ and on $\eta$. Let $s \in \opint{0,\infty}$ and let $\xi \in \opint{\pi/2,\pi/2+\delta}$. We must show that
\[
\frac{1}{2\pi} \int_{\Gamma(\j \, ;r;\eta)} \C_\alpha(\A) \, \j^{-1} e^{t\alpha} \de \alpha=
\frac{1}{2\pi} \int_{\Gamma(\j \, ;s;\xi)} \C_\alpha(\A) \, \j^{-1} e^{t\alpha} \de \alpha.
\]
Up to interchange $r$ with $s$, we may suppose that $r \leq s$. There are four cases to consider: $r<s$ and $\eta>\xi$, $r<s$ and $\eta=\xi$, $r<s$ and $\eta<\xi$, $r=s$ and $\eta>\xi$. We treat only the first. The proof of the other cases is similar. Hence suppose $r<s$ and $\eta>\xi$. Choose $h>s$ and define the open subset $U$ of $\CC_\j$ by setting
\[
U:=\big\{\rho e^{\theta\j} \in \CC_j \, : \, (\rho,\theta) \in \big(\opint{r,s} \times \opint{-\eta,\eta}\big) \cup \big(\opint{r,h} \times \opint{\xi,\eta}\big)\big\}.
\] 
Evidently, the closure of $U$ is contained in $\Sigma_\delta \cap \CC_\j$ (and hence in $\rho_\s(\A) \cap \CC_\j$) and the boundary of $U$ can be decomposed as the following formal sum of $C^1$-paths of $\rho_\s(\A) \cap \CC_\j \,$: 
\[
\Gamma(\j \, ;s,h;\xi)+\ce(\j \, ;h;\xi,\eta)-\Gamma(\j \, ;r,h;\eta)+\ce(\j \, ;h;-\eta,-\xi).
\]
By Corollary \ref{cor:inv-int-A}, we have that
\begin{align*}
&\frac{1}{2\pi} \int_{\Gamma(\j \, ;r,h;\eta)} \C_\alpha(\A) \, \j^{-1} e^{t\alpha} \de \alpha-
\frac{1}{2\pi} \int_{\Gamma(\j \, ;s,h;\xi)} \C_\alpha(\A) \, \j^{-1} e^{t\alpha} \de \alpha\\
&=\frac{1}{2\pi} \int_{\ce(\j \, ;h;\xi,\eta)+\ce(\j \, ;h;-\eta,-\xi)} \C_\alpha(\A) \, \j^{-1} e^{t\alpha} \de \alpha.
\end{align*}
In this way, since
\begin{align*}
\left\| \int_{\ce(\j \, ;h;\xi,\eta)+\ce(\j \, ;h;-\eta,-\xi)}\C_\alpha(\A) \j^{-1} \de \alpha e^{t\alpha} \right\| 
&=\left\|\int_{\clint{-\eta,-\xi} \cup \clint{\xi,\eta}} \C_{h e^{-\theta\j}}(\A) \, \j^{-1} \j \, he^{-\theta \j} e^{the^{-\theta\j}}\de \theta \right\| \\
&\leq 2\int_\xi^\eta Me^{th\cos(\xi)} \de\theta \to 0 \quad \text{as $h \to \infty$},
\end{align*}
we infer that
\begin{align*}
&\frac{1}{2\pi} \int_{\Gamma(\j \, ;r;\eta)} \C_\alpha(\A) \, \j^{-1} e^{t\alpha} \de \alpha=
\lim_{h \rightarrow \infty}\frac{1}{2\pi} \int_{\Gamma(\j \, ;r,h;\eta)} \C_\alpha(\A) \, \j^{-1} e^{t\alpha} \de \alpha\\
&=\lim_{h \rightarrow \infty}\frac{1}{2\pi} \int_{\Gamma(\j \, ;s,h;\xi)} \C_\alpha(\A) \, \j^{-1} e^{t\alpha} \de \alpha=\frac{1}{2\pi} \int_{\Gamma(\j \, ;s;\xi)} \C_\alpha(\A) \, \j^{-1} e^{t\alpha} \de \alpha,
\end{align*}
as desired. Thanks to the independence of integral in  \eqref{espressione integrale per il risolvente1} from $r$ and $\eta$ just proved, we can choose $r=1/t$ and $\eta = \delta/2$ in \eqref{first estimate for U}, obtaining the following uniform estimate in $t>0$:
\begin{align}
\norma{\T(t)}{} & \le \frac{M}{2\pi}\int_{-\eta}^\eta e^{\cos(\theta)} \de \theta+\frac{M}{\pi} \int_{1/t}^{\infty} \frac{e^{t\rho\cos(\delta/2)}}{\rho}\de \rho \notag \\
&=\frac{M}{2\pi}\int_{-\eta}^\eta e^{\cos(\theta)} \de \theta+\frac{M}{2\pi}\int_{-\cos(\eta/2)}^{\infty} \frac{e^{-r}}{r} \de r. \notag
\end{align}

Let us complete the proof by showing that $\frac{\de\ }{\de t}\T(t)x=\A\T(t)x$ if $t>0$. Let us observe that
\begin{equation}\label{int e^talpha = 0}
\int_{\Gamma(\j \,;r;\eta)} e^{t\alpha} \de \alpha=0.
\end{equation}
In fact, using the classical Cauchy formula, we have that
\[
\left\|\int_{\Gamma(\j \,;r,h;\eta)} e^{t\alpha} \de \alpha\right\|=\left\|-\int_{\ce(\j \,;h;\eta,2\pi-\eta)} e^{t\alpha} \de \alpha\right\| \leq 2(\pi-\eta)e^{th\cos(\eta)}h \to 0 \quad \text{as $h \to \infty$}.
\]
Moreover let us observe that $\int_{\Gamma(\j \, ;r;\eta)} \C_\alpha(\A) \alpha \, \j^{-1} e^{t\alpha} \de \alpha$ is absolutely convergent
and $\A\C_\alpha(\A) = \C_\alpha(\A)\alpha -  \Id$, by Lemma \ref{L:Calpha - AC = I}, hence from \eqref{int e^talpha = 0} we infer that 
$\int_{\Gamma(\j \, ;r;\eta)} \A\C_\alpha(\A) \, \j^{-1} e^{t\alpha} \de \alpha$ is also convergent. Therefore differentiating under the integral sign, using \eqref{P:Calpha - AC = I}, the closedness of $\A$, and taking into account 
the definition of line integral given in Appendix-Section \ref{S:appendix} (see formula \eqref{formula per bochner line int}), we get
\begin{align}
\frac{\de\ }{\de t}\T(t)x
&=\frac{1}{2\pi} \left(\int_{\Gamma(\j \, ;r;\eta)} \C_\alpha(\A) \alpha \, \j^{-1} e^{t\alpha} \de \alpha\right) x \notag \\
&=\frac{1}{2\pi} \left(\int_{\Gamma(\j \, ;r;\eta)} \A\C_\alpha(\A) \, \j^{-1} e^{t\alpha} \de \alpha\right) x+  
  \frac{1}{2\pi} \, \j^{-1}\int_{\Gamma(\j \,;r;\eta)}  e^{t\alpha}\de \alpha \ x \notag \\ 
&=\A  \left(\frac{1}{2\pi} \int_{\Gamma(\j \,;r;\eta)} \C_\alpha(\A) \, \j^{-1} e^{t\alpha} \de \alpha\right) x=\A\T(t)x. \notag
\end{align}
\end{proof}

Let us observe that whereas the first part of the proof of the preceding lemma is an adaptation of a classical argument to our general setting, we did not prove that $\T(t)$ is a right linear operator semigroup, indeed the arguments based on Fubini theorem or on the inversion of the Laplace transform do not work in the noncommutative setting. Therefore we have to follow a different procedure and we will deduce that $\T(t)$ is a semigroup from the fact that $\T(t)x$ solves the first order differential equation
$u'(t) = \A u(t)$. 

\begin{Lem}\label{Clambda come integrale}
Let $D(\A)$ be a dense right $\alg$-submodule of $X$, let $\A : D(\A) \lra X$ be a spherical $\delta$-sectorial operator for some $\delta \in \opint{0,\pi/2}$, let $r \in \opint{0,\infty}$ and let $\eta \in \opint{\pi/2,\pi/2+\delta}$. Then we have
\begin{equation}\label{espressione integrale per il risolvente}
\C_\lambda(\A) = \frac{1}{2\pi} \int_{\Gamma(\j \,;r;\eta)} \C_\alpha(\A) \, \j^{-1} (\lambda - \alpha)^{-1}\de \alpha \qquad \forall \lambda>r.
\end{equation}
\end{Lem}
\begin{proof}
Choose $\lambda \in \RR$ with $\lambda>r$. First of all we observe that 
\begin{align}
  & \sp \frac{1}{2\pi} \int_{\Gamma(\j\,;r;\eta)} \C_\alpha(\A) \, \j^{-1}(\lambda - \alpha)^{-1} \de \alpha=\frac{1}{2\pi} \int_{-\eta}^\eta \C_{r e^{\theta\j}}(\A) re^{\theta \j} (\lambda - r e^{\theta\j})^{-1} \de \theta \notag \\
&\sp-\frac{1}{2\pi} \int_{-\infty}^{-r} \C_{-\rho e^{-\eta \j}}(\A) \, \j^{-1} e^{-\eta \j} (\lambda+\rho e^{-\eta\j})^{-1}\de \rho\ +\frac{1}{2\pi} \int_r^{\infty} \C_{\rho e^{\eta \j}}(\A) \, \j^{-1} e^{\eta \j}(\lambda - \rho e^{\eta\j})^{-1}\de \rho. \notag
\end{align}
The convergence of the last two integrals follows from the estimate
\begin{align*}
&\left\|\int_{-\infty}^{-r} \C_{-\rho e^{-\eta \j}}(\A) \, \j^{-1} e^{-\eta \j} (\lambda+\rho e^{-\eta\j})^{-1}\de \rho\right\|\\
&+\left\|\int_r^\infty \C_{\rho e^{\eta \j}}(\A) \, \j^{-1}e^{\eta \j}(\lambda - \rho e^{\eta\j})^{-1}\de \rho \right\| \leq 2\int_r^\infty \frac{M}{\rho|\lambda -\rho e^{\eta\j}|} \de \rho.
\end{align*}
Let $h,R \in \RR$ such that $h>\lambda+1$ and $1/R=\min\{\lambda-r,1\}/2$. Define an open subset $U_h$ of $\CC_\j$ and a closed subset $D_R$ of $\CC_\j$ contained in $U_h$ as follows:
\[
U_h:=\big\{\rho e^{\theta\j} \in \CC_\j \, : \, (\rho,\theta) \in \opint{r,h} \times \opint{-\eta,\eta}\big\}
\; \; \mbox{ and } \; \;
D_R:=\big\{\alpha \in \CC_\j \, : \, |\alpha-\lambda| \leq 1/R\big\}.
\]
Observe that the boundary of $U \setminus D_R$ decomposes as the formal sum $\ce(\j \, ;h;-\eta,\eta)-\Gamma(\j \, ;r,h;\eta)-\partial D_R$ of $C^1$-paths of $\CC_\j$. Moreover, since $\A$ is $\delta$-sectorial, the closure of $U$ is contained in $\rho_\s(\A) \cap \CC_\j$. In this way, Corollary \ref{cor:inv-int-A} implies that
\begin{align*}
&\frac{1}{2\pi} \int_{\Gamma(\j \,;r,h;\eta)} \C_\alpha(\A) \, \j^{-1} (\lambda - \alpha)^{-1}\de \alpha\\
&=-\frac{1}{2\pi} \int_{\partial D_R} \C_\alpha(\A) \, \j^{-1} (\lambda - \alpha)^{-1}\de \alpha +\frac{1}{2\pi} \int_{\ce(\j \, ;h;-\eta,\eta)} \C_\alpha(\A) \, \j^{-1} (\lambda - \alpha)^{-1}\de \alpha.
\end{align*}
On the other hand, we have that
\begin{align}
&\left\|\int_{\ce(\j \, ;h;-\eta,\eta)} \C_\alpha(\A) \, \j^{-1} (\lambda - \alpha)^{-1}\de \alpha \right\| \notag \\
&=\left\| \int_{-\eta}^\eta \C_{h e^{\theta\j}}(\A) \j^{-1} \, \j \, h e^{\theta \j} a (\lambda - he^{-\theta \j})^{-1}\de \theta\right\| \notag \\
& \leq  \int_{-\eta}^\eta \frac{M}{|\lambda - h e^{-\theta \j}|} \de \theta \le \frac{2\eta M}{h-\lambda} \to 0 \quad \text{as $h \to \infty$} \notag
\end{align}
and hence
\begin{align*}
&\frac{1}{2\pi} \int_{\Gamma(\j \,;r;\eta)} \C_\alpha(\A) \, \j^{-1} (\lambda - \alpha)^{-1}\de \alpha=\lim_{h \rightarrow \infty}\frac{1}{2\pi} \int_{\Gamma(\j \,;r,h;\eta)} \C_\alpha(\A) \, \j^{-1} (\lambda - \alpha)^{-1}\de \alpha\\
&=-\frac{1}{2\pi} \int_{\partial D_R} \C_\alpha(\A) \, \j^{-1} (\lambda - \alpha)^{-1}\de \alpha.
\end{align*}
Define $\B:=-\C_\lambda(\A) \in \Lin^\r(X)$ and the map $\Phi:\alg \setminus \{\lambda\} \funzione \alg$ by $\Phi(\alpha):=(\alpha-\lambda)^{-1}$. By \eqref{relation for rho(A) and rho(B)}, we know that 
$\sigma_\s(\B) \cap \CC_\j \subseteq \Phi(\CC_j \setminus D_R)=B_\j(R)$. Now, using the change of variable $\beta=\Phi(\alpha)$ and Lemmas \ref{L:identity as integral of the resolvent} and \ref{L:lemma su Phi}, we get
\begin{align}
-\frac{1}{2\pi} \int_{\partial D_R} \C_\alpha(\A) \, \j^{-1} (\lambda - \alpha)^{-1}\de \alpha & = -\frac{1}{2\pi} \int_{\partial D_R}\B \C_{\Phi(\alpha)}(\B) \, \j^{-1} \Phi(\alpha)^2\de \alpha \notag \\
& = \B \frac{1}{2\pi} \int_{-\partial B_\j(R)} \C_\beta(\B) \, \j^{-1} \de \beta=-\B. \notag
\end{align}
In conclusion, we have that
\[
\frac{1}{2\pi} \int_{\Gamma(\j \,;r;\eta)} \C_\alpha(\A) \, \j^{-1} (\lambda - \alpha)^{-1}\de \alpha=-\B=\C_\lambda(\A). 
\]
The proof is complete.
\end{proof}

We are now in position to prove our main theorem about spherical sectorial operator.

\begin{Thm}
Let $D(\A)$ be a dense right $\alg$-submodule of $X$ and let $\A : D(\A) \lra X$ be a spherical $\delta$-sectorial operator for which there exists $M \in \opint{0,\infty}$ such that
\begin{equation}\label{eq:estimate}
\norma{\C_\alpha(\A)}{} \le \frac{M}{|\alpha|} \qquad \forall \alpha \in \Sigma_\delta.
\end{equation}
Fix $\j \in \su_\alg$ and define the map $\T : \clsxint{0,\infty} \funzione \Lin^r(X)$ by setting
\begin{align}
  & \T(0) := \Id, \\
  & \T(t) := \frac{1}{2\pi} \int_{\Gamma(\j \,;r;\eta)} \C_\alpha(\A) \, \j^{-1} e^{t\alpha}\de \alpha \label{T(t), t > 0} \qquad \forall t > 0,
\end{align}
for some $r \in \opint{0,\infty}$ and $\eta \in \opint{\pi/2,\pi/2+\delta}$. Then $\T$ is the semigroup  generated by $\A$ and the following equality holds:
\begin{equation} \label{eq:int-C}
\C_\lambda(\A) = \R_\lambda(\A) = \int_0^\infty e^{-t\lambda}\T(t) \de t \qquad \forall \lambda > 0.
\end{equation}
As a consequence, the integral in \eqref{T(t), t > 0} is independent of $\j \in \su_\alg$ and $\T$ is analytic in time.
\end{Thm}

\begin{proof}
By Lemma \ref{L:indipendenza di U(t) dalla curva}, we know that the integral in \eqref{T(t), t > 0} is convergent in $\Lin^\r(X)$, does not depend on the choice of $r>0$ and $\eta \in \opint{\pi/2,\pi/2+\delta}$ and there exists a constant $C \ge 0$ such that
\begin{equation}
  \norma{\T(t)}{} \le C \qquad \forall t \ge 0.
\end{equation}
Fix $\lambda > 0$ and define $\Gamma:=\Gamma(\j \,;\lambda/2;\eta)$. Choose $L>0$. Since $\lambda \in \RR$ and $\alpha \in \CC_\j$ commute, the Fubini theorem, the fundamental theorem of calculus and Lemma \ref{Clambda come integrale} ensure that
\begin{align}\label{int_0^T e^-tl T(t)}
\int_0^L e^{-t\lambda} \T(t) \de t 
&=\frac{1}{2\pi} \int_\Gamma \int_0^L \C_\alpha(\A) \, \j^{-1} e^{t(\alpha - \lambda)} \de \alpha \notag \\
&=\frac{1}{2\pi} \int_\Gamma \C_\alpha(\A) \, \j^{-1} e^{L(\alpha - \lambda)}(\alpha - \lambda)^{-1} \de \alpha-
\frac{1}{2\pi} \int_\Gamma \C_\alpha(\A) \, \j^{-1} (\alpha - \lambda)^{-1}\de \alpha \notag \\
&=\frac{1}{2\pi} \int_\Gamma \C_\alpha(\A) \, \j^{-1} e^{L(\alpha - \lambda)}(\alpha - \lambda)^{-1}\de \alpha + \C_\lambda(\A). \notag
\end{align}
Since
\begin{align}
&\left\| \int_\Gamma \C_\alpha(\A) \, \j^{-1} e^{L(\alpha - \lambda)}(\alpha - \lambda)^{-1}\de \alpha \right\| \notag\\
&\le \int_\Gamma \frac{M}{|\alpha|} \frac{e^{L(\re(\alpha) - \lambda)}}{|\alpha - \lambda|} |\de \alpha| \le \int_\Gamma \frac{M}{|\alpha|} \frac{e^{-L\lambda/2}}{|\alpha - \lambda|} |\de \alpha| \to 0 \quad \text{as $L \to \infty$}, \notag
\end{align}
formula \eqref{eq:int-C} is proved. Differentiating under the integral sign in the mentioned formula, we find
\begin{equation}
\frac{\de^{n-1}\ }{\de\lambda^{n-1}} \R_\lambda(\A) = (-1)^{n-1}\int_0^\infty \T(t) t^{n-1}e^{-t\lambda} \de t
\end{equation}
Since $\frac{\de^{n-1}\ }{\de\lambda^{n-1}} \R_\lambda(\A) = (-1)^{n-1}(n-1)! \R_\lambda(\A)^n$, formula \eqref{resolvent equation} implies that
\begin{equation}
\C_\lambda(\A)^n = \R_\lambda(\A)^n = \frac{1}{(n-1)!} \int_0^\infty\T(t) t^{n-1}e^{-t\lambda} \de t \qquad \forall \lambda > 0.
\end{equation}
It follows that
\begin{equation}
  \norma{\C_\lambda(\A)^n}{} \le \frac{1}{(n-1)!} \int_0^\infty  C t^{n-1}e^{-t\lambda} \de t = \frac{C}{\lambda^n}  \le
  \frac{C}{(\lambda - r)^n}
  \qquad \forall \lambda > r.
\end{equation}
Hence we can apply Theorem \ref{T:A-generation thm} and deduce that $\A$ generates a strongly continuous semigroup $\U(t)$ such that 
\[
\norma{\U(t)}{} \le C e^{rt}.
\]
Moreover, for every $x \in D(\A)$, $\U(\cdot)x$ is the unique solution of the Cauchy problem
\begin{align}
  & u'(t) = \A u(t), \qquad t > 0, \\
  & u(0) = x.
\end{align}
On the other hand, we know by Lemma \ref{L:indipendenza di U(t) dalla curva} that
\begin{equation}
  \frac{\de\ }{\de t} \T(t)x = \A \T(t)x \qquad \forall t > 0,
\end{equation}
therefore by Theorem \ref{T:semigroups and diff eq-A}
\[
  \T(t)x = \U(t)x \qquad \forall t > 0, \quad \forall x \in D(\A)
\]
which yields $\T = \U$.
\end{proof}


\section{Appendix}\label{S:appendix}
In this short section, we review the notion of line integral for vector functions with values in a Banach $\alg$-bimodule.

Let $\alg$ be a real algebra satisfying \eqref{eq:assumption-1} and \eqref{eq:assumption-2}.
Let $V$ be a Banach $\alg$-bimodule and let $\gamma : \clint{a,b} \funzione \alg$ be a piecewise $C^1$ function. Denote by 
$|\gamma| \subseteq \alg$ the
image 
of $\gamma$. Given continuous functions $F :|\gamma| \funzione V$ and $f:|\gamma| \lra \alg$, we set
\begin{equation}
  \int_\gamma F(\alpha) \de \alpha \, f(\alpha):= 
  \int_a^b F(\gamma(s)) \gamma'(s) f(\gamma(s)) \de s,
\end{equation}
where the last integral is a Bochner integral. Let us recall briefly the definition of the Bochner integral in this new context. Given $m \in \enne$, $S_k \in V$ and Borel subsets $B_k$ of $\clint{a,b}$ for $k=1,\ldots,m$, we consider the 
measurable 
step function 
$S: \clint{a,b} \funzione V$ defined by
\[
S :=\sum_{k=1}^m S_k\chi_{B_k},
\]
where $\chi_{B_k}$ is the characteristic function of $B_k$ in $\clint{a,b}$. We set
\[
 \int_a^b S(s) \de s:=\sum_{k=1}^m\mr{m}(B_k)S_k,
\]
where $\mr{m}(B_k)$ is the Lebesgue measure of $B_k$. Denote by $\mi{St}(\clint{a,b};V)$ the set of measurable step functions on $\clint{a,b}$ with values in $V$, endowed with the $L^1$-seminorm $\|S\|:=\int_a^b\norma{S(s)}{V} \de s$, where $\norma{\cdot}{V}$ is the norm of~$V$. The map $\mi{St}(\clint{a,b};V) \lra V:S \longmapsto \int_a^b S(s) \de s$ is $\erre$-linear and continuous. Therefore, it admits a continuous extension
$I:L^1(\clint{a,b};V) \funzione V$ to the completion of $\mi{St}(\clint{a,b};V)$. In the commutative setting this extension is usually  called \emph{Bochner integral}. Now one can define
\begin{equation}
 \int_a^b F(\gamma(s)) \gamma'(s) f(\gamma(s)) \de s:= I\big((F \circ \gamma)\gamma'(f \circ \gamma)\big).
\end{equation}
For simplicity, if $f$ coincides with the function $\mathbf{1}:|\gamma| \lra \alg$ constantly equal to $1$, then we write $\int_\gamma F(\alpha)\de \alpha$ in place of $\int_\gamma F(\alpha)\de \alpha \, \mathbf{1}(\alpha)$. Evidently, we have: $\int_\gamma F(\alpha)\de \alpha=\int_a^bF(\gamma(s))\gamma'(s) \de s$. We underline that, if $\gamma'(s)$ and $f(\gamma(s))$ commute a.e. in $s \in \clint{a,b}$, then it holds:
\[
\int_\gamma F(\alpha)\de \alpha \, f(\alpha)=\int_\gamma F(\alpha)f(\alpha)\de \alpha.
\]

We warn the reader that, if $\alg$ is associative, but not commutative, and if $X$ is a Banach $\alg$-bimodule, then in general, for $\B : \Omega \funzione \Lin^r(X)$ continuous and $x \in X$, we have
\begin{equation}
  \left(\int_\gamma \B(\alpha) \de \alpha \right)(x) \neq \int_\gamma \B(\alpha)(x) \de \alpha,
\end{equation}
where on the left-hand side there is the Bochner integral of a function with values in $V = \Lin^r(X)$, and on the right-hand side
we have the integral of a function with values in $V = X$. In fact, we have the following identities:
\begin{equation}\label{formula per bochner line int}
  \left(\int_\gamma \B(\alpha) \de \alpha\right)(x) = 
  \int_a^b \big(\B(\gamma(s))\gamma'(s)\big)(x) \de s
\end{equation}
and
\begin{equation}
   \int_\gamma \B(\alpha)(x) \de \alpha = 
  \int_a^b \big(\B(\gamma(s))(x)\big) \gamma'(s) \de s.
\end{equation}



\end{document}